\documentclass[11pt]{article}
\usepackage{anysize,color}
\usepackage[utf8]{inputenc}
\usepackage[colorlinks=false]{hyperref}  
\usepackage[pdftex]{graphicx}
\usepackage{float}
\usepackage{diagrams}
\usepackage{amsfonts,amssymb,amsmath}
\usepackage{amsthm}    
\usepackage{url}
\usepackage{paralist}
\usepackage{tikz}
    \usetikzlibrary{decorations.markings,arrows,shapes.callouts}
\usepackage{rotating} 
\usepackage[mathscr]{euscript}

\newtheorem{theorem}{Theorem}[section]
\newtheorem{proposition}[theorem]{Proposition}
\newtheorem{lemma}[theorem]{Lemma} 
\newtheorem{conjecture}[theorem]{Conjecture}
\newtheorem{corollary}[theorem]{Corollary} 
\theoremstyle{definition}
\newtheorem{definition}[theorem]{Definition}
\theoremstyle{remark}
\newtheorem{example}[theorem]{Example} 
 
\newtheorem{remark}[theorem]{Remark} 
   \newcommand\R{\mathbb R}
   \newcommand\C{\mathbb C}
   \newcommand\Z{\mathbb Z}
   \newcommand\CP{\mathbb C\mathrm P}
   \newcommand\Symm{\mathfrak S}
   \newcommand\oo{\mathfrak o}
   \newcommand\ii{{\bf i}} 
   \newcommand\sgn{{\rm sgn\,}}
   
   \newcommand\less[1]{{<_{#1}}} 
   
   \newcommand\VoronoiCell{\mathscr V}
   \newcommand\PowerCell{\mathscr P}
   \newcommand\ZZ{\mathscr Z}
   
    \newcommand\PosetOfStratification{\mathrm S}
    \newcommand\CellComplexFullSphere{\mathscr S}
    \newcommand\CellComplexFullBall{\mathscr B}
    \newcommand\SimplicialFullBall{{\rm sd}\,\mathscr B}
            
    \newcommand\PosetOfComplement{\mathrm F}
    \newcommand\CellComplexComplement{\mathscr F}
    \newcommand\SimplicialComplement{{\rm sd}\,\mathscr F}
    \newcommand\EAP{\mathrm{EAP}}
\newcommand\EMP{\mathrm{EMP}}

\begin{document}

\title{Convex Equipartitions via Equivariant Obstruction Theory\thanks{%
    The research leading to these results has received funding from the European Research
    Council under the European Union's Seventh Framework Programme (FP7/2007-2013) /
    ERC Grant agreement no.~247029-SDModels. The first author was also supported by the grant ON 174008 of the Serbian
    Ministry of Education and Science.}}
\author{%
Pavle V. M. Blagojevi\'{c}\\
Mathemati\v{c}ki Institut SANU\\Knez Mihailova 36\\11001 Beograd, Serbia\\
\url{pavleb@mi.sanu.ac.rs}\\ and \\ 
Institut f\"ur Mathematik, FU Berlin\\Arnimallee 2\\14195 Berlin, Germany\\
\url{blagojevic@math.fu-berlin.de}
\and \setcounter{footnote}{0}
G\"unter M. Ziegler\\
Institut f\"ur Mathematik, FU Berlin\\Arnimallee 2\\14195 Berlin, Germany\\
\url{ziegler@math.fu-berlin.de}}
\date{{\small September 2012; revised, January 2013}}
\maketitle

\begin{abstract}
 We describe a regular cell complex model for the configuration space $F(\R^d,n)$. Based on this,
 we use Equivariant Obstruction Theory to prove the prime power case of the
 conjecture by Nandakumar and Ramana Rao that every polygon can be partitioned into $n$
 convex parts of equal area and perimeter.  
\end{abstract}

\section{Introduction} 

R. Nandakumar and N. Ramana Rao in June 2006 posed the following problem, 
which in~\cite{BaranyBlagojevicSzuecs} was characterized as ``interesting and annoyingly resistant'':

\begin{conjecture}[Nandakumar \& Ramana Rao \cite{Nandakumar06}]  
    Every convex polygon $P$ in the plane can be partitioned into any prescribed number $n$ of
    convex pieces that have equal area and equal perimeter.
\end{conjecture}

The conjecture was proved by Nandakumar \& Ramana Rao \cite{NandakumarRamanaRao12} for $n=2$, where it 
follows from the intermediate value theorem, that is, the Borsuk--Ulam theorem for maps $f:S^1\rightarrow S^0$;
the same authors also gave elementary arguments for the result for $n=2^k$.
For $n=3$ the conjecture was proven by Bárány, Blagojevi\'c \& Sz\H{u}cs \cite{BaranyBlagojevicSzuecs} via advanced
topological methods.

For general $n$, it was noted by 
Karasev \cite{Karasev-equipartition} and Hubard \& Aronov \cite{HubardAronov-equipartition}
that the conjecture would follow from the non-existence of an $\Symm_n$-equivariant map  
$F(\R^2,n)\rightarrow S(W_n)$,
where $F(\R^2,n)$ denotes the configuration space of $n$ distinct labeled points in the plane, that is,
the set of all real $2\times n$ matrices with pairwise distinct columns, and
$W_n:=\{(x_1,\dots,x_n)\in\R^n:x_1+\dots+x_n=0\}$ is the set of all row vectors in $\R^n$ with vanishing sum of coordinates.
Both $F(\R^2,n)$ and $W_n$ have the obvious action of~$\Symm_n$ by permuting columns. 
Indeed, assume that a convex polygon $P\subset\R^2$ yields a counterexample to the
Nandakumar and Ramana Rao conjecture for some $n\ge2$. Then we get an $\Symm_n$-equivariant map
\[
\EAP(P,n)\ \longrightarrow_{\Symm_n}\  S(W_n)
\]
from a suitably defined (metric) \emph{space of equal area partitions} of~$P$, denoted $\EAP(P,n)$,
to the sphere $S(W_n)$: It maps the partition $(P_1,\dots,P_n)$ to the vector of perimeters
$(p(P_1),\dots,p(P_n))$, then subtracts the average perimeter $\frac1n(p(P_1)+\dots+p(P_n))$ from each component of this vector,
and finally normalizes the resulting vector to have length~$1$.
Moreover, mapping each piece $P_i$ to its barycenter results in a continuous $\Symm_n$-equivariant map 
\[
\EAP(P,n)\ \longrightarrow_{\Symm_n}\  F(\R^2,n),
\]
while the theory of Optimal Transport (see Section~\ref{sec:transport}) yields a continuous $\Symm_n$-equivariant map 
\[
F(\R^2,n)\ \longrightarrow_{\Symm_n}\ \EAP(P,n).
\]
Thus the existence of an $\Symm_n$-equivariant map
$\EAP(P,n)\longrightarrow_{\Symm_n}S(W_n)$   
is equivalent to the existence of an $\Symm_n$-equivariant map
\[
F(\R^2,n)\ \longrightarrow_{\Symm_n}\  S(W_n).
\]
Similarly, for $d$-dimensional generalizations of the Nandakumar and Ramana Rao problem, asking for partition of a 
$d$-dimensional polytope, or a sufficiently continuous measure on $\R^d$, into $n$ convex pieces on which
$d-1$ continuous functions are equalized, one would try to prove the
non-existence of $\Symm_n$-equivariant maps 
\[
F(\R^d,n)\ \longrightarrow_{\Symm_n}\  S(W_n^{\oplus(d-1)}).
\]
First Karasev \cite{Karasev-equipartition} and later Hubard \& Aronov \cite{HubardAronov-equipartition}
proposed to show this non-existence  
for prime power $n=p^k$ and $d\ge2$ based on a cell decomposition of the one-point compactification
of the configuration space that appears in work by Fuks \cite{Fuks70} and Vassiliev \cite{Vassiliev88}.
Thus they have to deal with twisted Euler classes on non-compact manifolds resp.\ on compactifications
that are not manifolds; no complete treatment for this has been given yet.

Indeed, to see that the situation is not easily handled with Euler class arguments,  
let $E\to M$ be a vector bundle over an orientable open manifold $M$, 
let $E_0$ be the corresponding zero section, 
and let $s:M\to E$ be a generic cross section.
This means that $s(M)$ intersects $E_0$ transversally, and consequently the singular set $s(M)\cap E_0$ is a manifold.
The fundamental class of the singular set lives in the homology with closed supports 
(or Borel--Moore homology) $H_*^{\mathrm{cl}}(M)$, \cite{Borel--More60}.
The homology with closed supports is not a homotopy invariant of the space \cite[Remark on p.~48]{Prasolov05}.
Therefore, the Poincar\'e duality isomorphism, which relates the homology with closed supports
$H_*^{\mathrm{cl}}(M)$ with the relative cohomology of the one-point compactification
$H^*(\Hat{X},\mathrm{pt})$, needs not map the fundamental class of the singular set to the Euler class of the vector bundle $E\to M$.
In particular, the fundamental classes of singular sets of two different generic cross sections do not need to coincide in $H_*^{\mathrm{cl}}(M)$.

In this paper we obtain the desired result, while avoiding these difficulties, via Equivariant Obstruction Theory.
Moreover, this approach yields an ``if and only if'' statement:

\begin{theorem}\label{thm:main}
    For $d,n\ge2$, an $\Symm_n$-equivariant map 
    \[
    F(\R^d,n)\ \longrightarrow_{\Symm_n}\  S(W_n^{\oplus(d-1)})
    \]
    exists if and only if $n$ is not a prime power.
\end{theorem}

For this, we describe in Section~\ref{sec:cellcomplex} 
a beautiful $\Symm_n$-equivariant $(d-1)(n-1)$-dimensional cell complex with $n!$ vertices
and $n!$ facets that is an $\Symm_n$-equivariant strong deformation retract of $F(\R^d,n)$.
For $d=2$ this complex is implicit in work by Deligne \cite{Deligne72} and explicit in 
Salvetti's work \cite{Salvetti}. For $d\ge2$ we obtain it by specialization of a construction
due to Björner \& Ziegler \cite{BjornerZiegler-combstrat}.

Theorem~\ref{thm:main} in particular resolves the original problem by Nandakumar \& Ramana Rao
for prime powers $n$. 

\begin{theorem}\label{thm:equipartition}
    Let $d\ge2$ and let $\mu_d$ be an absolutely continuous probability measure on $\R^d$
    (given by a nonnegative Lebesgue integrable density function with convex support).
    
\noindent
    If $n$ is a prime power, then for any $d-1$ continuous functions $f_i$, $1\le i\le d-1$,
    on the space of nonempty convex bodies there is a partition
    of $\R^d$ into $n$  convex regions $C_1,\dots,C_n\subset\R^d$ that equiparts~$\mu_d$,
    \[
    \mu_d(C_1)=\dots=\mu_d(C_n)=\tfrac1n,\hspace{30mm}
    \]
    and that at the same time equalizes the functions $f_1,\dots,f_{d-1}$:
    \[ 
    f_{i}(C_1)=\dots=f_{i}(C_n)\qquad\text{for }1\le i\le d-1.
    \]
\end{theorem}

We note that this solves the Nandakumar and Ramana Rao Conjecture for all $d\ge2$ and all prime powers $n\ge2$,
and the problem is quite easily solved for $d=1$, while the original problem remains open already for $d=2$ and $n=6$.

Nevertheless, from Theorem~\ref{thm:main} we also get a much stronger theorem for the prime power case, for continuous functions
on the \emph{space of equal measure partitions}, whose components need not depend only on the corresponding pieces.
This generalization \emph{fails} if $n$ is not a prime power.

\begin{theorem}\label{thm:equipartition2}
    Let $d\ge2$ and let $\mu_d$ be an absolutely continuous probability measure on $\R^d$.
    
\noindent
    If $n$ is a prime power, then for any $d-1$ continuous $\Symm_n$-equivariant functions $F_i=(f_{1,i},\dots,f_{n,i})$, $1\le i\le d-1$,
    on the space $\EMP(\mu_d,n)$ of partitions of $\R^d$ into $n$ convex regions $C_1,\dots,C_n\subset\R^d$ that equipart~$\mu_d$,
    \[
    \mu_d(C_1)=\dots=\mu_d(C_n)=\tfrac1n,\hspace{30mm}
    \]
    there is such a partition $(C_1,\dots,C_n)\in\EMP(\mu_d,n)$  
    which at the same time equalizes the components of each of the functions $F_1,\dots,F_{d-1}$:
    \[ 
    f_{1,i}(C_1,\dots,C_n)=\dots=f_{n,i}(C_1,\dots,C_n)\qquad\text{for }1\le i\le d-1.
    \]  
    If $n$ is not a prime power, then this is false: For any given $d\ge2$, \textbf{any} 
    continuous positive probability measure $\mu_d$ on~$\R^d$, and $n$ not a prime power,
    there are $\Symm_n$-equivariant continuous functions $F_1=(f_{1,1},\dots,f_{n,1}),\dots,F_{d-1}=(f_{1,n},\dots,f_{n,n})$ on the space  
	$\EMP(\mu_d,n)$ of equal measure partitions of $\R^d$ into $n$ convex polyhedra
	such that no such partition equalizes the components for all $d-1$ functions~$F_1,\dots,F_{d-1}$.
\end{theorem}

An example of a function $F_1=(f_{1,1},\dots,f_{n,1})$ that satisfies the conditions of Theorem~\ref{thm:equipartition2}
is given by
\[
f_{i,1}(C_1,\dots,C_n)\ :=\ \textrm{diam}(C_i)-\tfrac1n(\textrm{diam}(C_1)+\dots+\textrm{diam}(C_n)),
\]
as the components $f_{i,1}$ of this function $F_1$ do not depend on $C_1$ alone.

Our proof of Theorem~\ref{thm:main} relies on Equivariant Obstruction Theory
as formulated by tom Dieck \cite[Sec.~II.3]{tomDieck:TransformationGroups}. 
We happen to be in the fortunate situation where we have to deal only with primary obstructions,
for mapping a cell complex of dimension $(d-1)(n-1)$ to a sphere of dimension $(d-1)(n-1)-1$.
Moreover, the cell complex has a single orbit of maximal cells (“facets”), on which the
obstruction cocycle evaluates to~$1$. At the same time, we have $n-1$ orbits of cells of
dimension $(d-1)(n-1)-1$ (called “subfacets” or “ridges”); these orbits have 
cardinalities $\binom ni$. The proof is completed via a number theoretic fact about the
$n$-th row of the Pascal triangle:
\[
 \text{gcd}\Big\{\binom n1,\binom n2,\dots,\binom n{n-1}\Big\}\ =\ 
 \begin{cases} p & \text{if }n\text{ is a prime power, }n=p^k,\\
    1 & \text{otherwise.}
 \end{cases}
\]
This may be derived from the work by Legendre and Kummer on prime factors in factorials
and binomial coefficients, apparently was first proved by Ram \cite{Ram1909} in 1909,
extended by Joris et al.~\cite{JorisOestreicherSteinig}, and
was later rediscovered by A. Granville
(see Soulé \cite{Soule04} and Kaplan \& Levy~\cite{KaplanLevy04}).

\section{From Convex Partitions to Configuration Spaces}\label{sec:transport} 

In this section we briefly sketch the provenance of the $\Symm_n$-equivariant map 
\[
F(\R^d,n)\ \longrightarrow_{\Symm_n}\  \EMP(\mu_d,n)
\]
used in our proof of the first half of Theorem~\ref{thm:main} (following Hubard \& Aronov \cite{HubardAronov-equipartition}). 
It maps point configurations to convex partitions that 
are “regular” in the sense of Gelfand, Kapranov \& Zelevinsky \cite{GKZ} \cite[Lect.~4]{Ziegler-poly},
that is, which arise from piecewise-linear convex functions.
In Computational Geometry such partitions are known as \emph{generalized Voronoi diagrams} or \emph{power diagrams}.
The equivalence between partitions induced by Optimal Transport and those given by the domains of
linearity of a minimum of $n$ linear functions was proved by Aurenhammer \cite{Aurenhammer87},
and certainly also by others; see Siersma \& van Manen \cite{Siersma_van_Manen}.  
 Theorem~\ref{thm:equipartition} yields
Theorem~\ref{thm:main} as regular equipartitions are parametrized by
$F(\R^d,n)$. This remarkable observation
can be traced back to Minkowski \cite{Minkowski1897,Minkowski1903}. It was independently
developed in Optimal Transport (see Villani \cite[Chap.~2]{Villani03}) and in 
Computational Geometry (see Aurenhammer, Hoffmann \& Aronov \cite{AurenhammerHoffmannAronov92,AurenhammerHoffmannAronov98}).

A classical case of regular partitions is known 
in Computational Geometry as \emph{Voronoi diagrams} and in Number Theory as \emph{Dirichlet tesselations}.
These are obtained as $\VoronoiCell (x_1,\dots,x_n)= (C_1,\dots,C_n)$ via
\begin{eqnarray*}
    C_i&:=&\{x\in\R^d : \|x-x_i\|\le\|x-x_j\|\text{ for }1\le j\le n\}\\
       & =&\{x\in\R^d : \|x-x_i\|^2-\|x\|^2\le\|x-x_j\|^2-\|x\|^2\text{ for }1\le j\le n\}.
\end{eqnarray*}
Thus $C_i$ is the domain in $\R^d$ where the linear (!) function $f_i(x)=\|x-x_i\|^2-\|x\|^2$
yields the minimum among all functions $f_j(x)$ of the same type, 
and thus the Voronoi diagram is given by the domains of linearity
of the piecewise-linear convex function $f(x)=\min\{ \|x-x_i\|^2-\|x\|^2:1\le i\le n\}$.

For \emph{generalized Voronoi diagrams} or \emph{power diagrams} we introduce additional real
weights $w_1,\dots,w_n\in\R$, and set $\PowerCell(x_1,\dots,x_n;w_1,\dots,w_n)= (P_1,\dots,P_n)$ 
with
\[
P_i:=\{x\in\R^d : \|x-x_i\|^2-w_i\le\|x-x_j\|^2-w_j\text{ for }1\le j\le n\}.
\] 
As an additive constant does not change the partition, we may without loss of generality replace
$w_i$ by $w_i-\frac1n(w_1+\dots+w_n)$, and thus assume that $w_1+\dots+w_n=0$, that is, $(w_1,\dots,w_n)\in W_n$.

Again, the partition of space given by a generalized Voronoi diagram is a regular subdivision. However,
now some of the parts~$P_i$ may be empty, and $x_i\in P_i$ does not hold in general.
We refer to Aurenhammer \& Klein \cite{AurenhammerKlein00} 
and Siersma \& van Manen \cite{Siersma_van_Manen} for further background.

The major part of the following theorem (except for the continuity, which of course is essential for us)
was provided by Aurenhammer, Hoffmann \& Aronov \cite{AurenhammerHoffmannAronov92} in 1992. 
The 1998 journal version \cite{AurenhammerHoffmannAronov98} of their paper noted the connection to
Optimal Transport; thus, the proof can be done using optimal assignment and linear programming duality,
as developed (for this purpose) by Kantorovich in 1939 \cite{Kantorovich40}, with the
``double convexification'' trick. For a modern exposition
see Villani \cite[Chap.~2]{Villani03}. Evans~\cite{Evans_OptimalTransport} provides a 
nice account based on discretization. Alternatively,  
Gei\ss~et~al.\ \cite{GKPR}
obtain the optimal weights from minimizing a quadratic objective function. 
  
\begin{theorem}[Kantorovich 1939, etc.]
    Let $\mu_d$ be an absolutely continuous probability measure on~$\R^d$ 
    and $n\ge2$. 
    Then for any {$x_1,\dots,x_n\in\R^d$}
    with $x_i\neq x_j$ for all $1\le i<j\le n$ there are weights $w_1,\dots,w_n\in\R$, $w_1+\dots+w_n=0$,
    such that the power diagram $\PowerCell(x_1,\dots,x_n;w_1,\dots,w_n)$
    equiparts the measure~$\mu_d$. Moreover, the weights $w_1,\dots,w_n$  
    \begin{compactitem}[$\circ$]
        \item are unique,
        \item depend continuously on $x_1,\dots,x_n$, and
        \item can be characterized/computed via optimal assignment with respect to  
          quadratic cost functions.%
    \end{compactitem}
\end{theorem}  

In short, this theorem yields a continuous, $\Symm_n$-equivariant map
$F(\R^d,n)\rightarrow_{\Symm_n}\EMP(\mu_d,n)$. 
Moreover, mapping to the barycenters (with respect to the measure $\mu_d$) of the pieces
yields an $\Symm_n$-equivariant map
$\EMP(\mu_d,n)\rightarrow_{\Symm_n}F(\R^d,n)$.
Thus  the claim of Theorem~\ref{thm:equipartition2} is equivalent
to the fact that an $\Symm_n$-equivariant continuous map 
\[
  F(\R^d,n) \ \longrightarrow_{\Symm_n}\ W_n^{\oplus(d-1)}\setminus\{0\}.
\]
exists if and only if $n$ is not a prime power. This will be established in the following.
 
\section{A Cell Complex Model for $F(\R^d,n)$}\label{sec:cellcomplex}

In order to apply obstruction theory to $F(\R^d,n)$, we specialize the
cell complex models given for complements of real subspace arrangements 
by Björner \& Ziegler \cite{BjornerZiegler-combstrat}.
For $d=2$, the barycentric subdivision of this cell complex is treated implicitly in Deligne \cite{Deligne72}.
The cell complex for a general complement of a complexified hyperplane arrangements is
developed in more detail by Salvetti~\cite{Salvetti} and thus known as the ``Salvetti complex.'' 
These cell complex models are closely related to the cell complex structures described already in 1962 by 
Fox \& Neuwirth \cite{FoxNeuwirth62}, and later by Fuks \cite{Fuks70} 
for $S^{2n}=\R^{2n}\cup\{*\}\supset F(\R^2,n)$.
The analogous constructions for $F(\R^d,n)$ were done by Vassiliev \cite{Vassiliev:Discriminants}.
See also de Concini \& Salvetti \cite{DeConciniSalvetti1990} and Basabe et al.~\cite{BasabeEtAl},
as well as very recently Guisti \& Sinha \cite{GiustiSinha} and Ayala \& Hepworth~\cite{AyalaHepworth}.%

To obtain our cell complex model, we first retract $F(\R^d,n)$ by $\overline x_i:=x_i-\frac1n(x_1+\dots+x_n)$~to%
\[
\overline F(\R^d,n):=\{(x_1,\dots,x_n)\in\R^{d\times n}:x_1+\dots+x_n=0,\ x_i\neq x_j\text{ for }1\le i<j\le n\},
\]
which is the complement of an {essential} arrangement of linear subspaces of codimension $d$
in the vector space $W_n^{\oplus d}=\{(x_1,\dots,x_n)\in\R^{d\times n}:x_1+\dots+x_n=0\}$ of dimension $d(n-1)$.
(An arrangement of linear subspaces is \emph{essential} if the intersection of all the subspaces is
the origin.)
All the subspaces in the arrangement have codimension $d$ and all their intersections have codimensions that are multiples of~$d$,
so this is a \emph{$d$-arrangement} in the sense of 
Goresky \& MacPherson \cite[Sect.~III.4]{GoreskyMacPherson}. See Björner \cite{Bjorner-ecm} for more background on subspace arrangements.
 
To construct the \emph{Fox--Neuwirth stratification} of this essential arrangement of linear subspaces
(which, except for the ordering of the coordinates,  is equivalent
to the ``$s^{(1)}$-stratification'' of \cite[Sect.~9.3]{BjornerZiegler-combstrat}), we  
order the $n$ columns of the matrix $(x_1,\dots,x_n)\in W_n^{\oplus d}$ in lexicographic order. 
That is, we find a unique permutation $\sigma: i\mapsto\sigma_i$ ($1\le i\le n$), which we denote by $\sigma=\sigma_1\sigma_2\dots\sigma_n\in\Symm_n$,
such that for $1\le j<n$,
\begin{compactitem}[ $\bullet$]
    \item either for some $i_j$, $1\le i_j\le d$, we have
             $x_{i_j,\sigma_j}<x_{i_j,\sigma_{j+1}}$ with $x_{i',\sigma_j}=x_{i',\sigma_{j+1}}$ for all $i'<i_j$,
    \item or $x_{\sigma_j}=x_{\sigma_{j+1}}$; in this case we set $i_j=d+1$ and impose
            the extra condition that $\sigma_j<\sigma_{j+1}$.
\end{compactitem}
Then we assign to $(x_1,\dots,x_n)$ the combinatorial data 
\[
     (\sigma_1\less{i_1}\sigma_2\less{i_2}\cdots\less{i_{n-1}}\sigma_n),
\]     
where $\sigma=\sigma_1\sigma_2\dots\sigma_n\in\Symm_n$ is a permutation and
$\ii=(i_1,\dots,i_{n-1})\in\{1,\dots,d,d+1\}^{n-1}$ is a vector of coordinate indices.  

\begin{example}\label{ex:8points_permutation}
    Let $d=2$, $n=8$: The $8$ points $x_1,\dots,x_8\in\R^2$ in the left part of Figure~\ref{fig:figure1}
    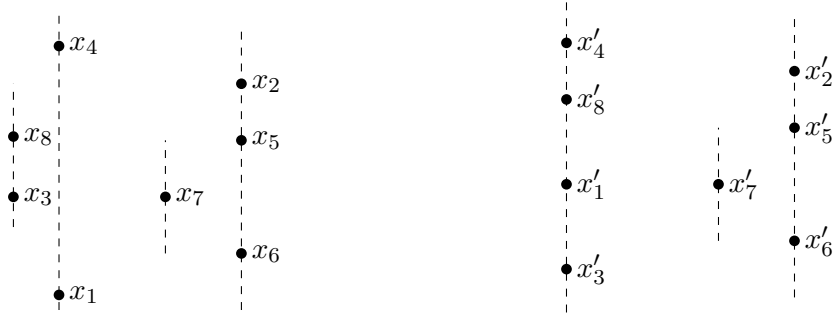
\begin{figure}[ht!]
        \begin{center}
            \medskip
        \begin{tikzpicture}
        	\draw[dashed] (0,0) -- (0,3.75);	
        	\draw[dashed] (-1,0.75) -- (-1,2.25);
        	\draw[dashed] (-2.4,0) -- (-2.4,4);
        	\draw[dashed] (-3,1.1) -- (-3,3);
        	\fill (0,0.75) circle (2pt) node[right] {$x_6$};
        	\fill (0,2.25) circle (2pt) node[right] {$x_5$};
        	\fill (0,3) circle (2pt) node[right] {$x_2$};

        	\fill (-1,1.5) circle (2pt) node[right] {$x_7$};

        	\fill (-2.4,0.2) circle (2pt) node[right] {$x_1$};
        	\fill (-2.4,3.5) circle (2pt) node[right] {$x_4$};

        	\fill (-3,1.5) circle (2pt) node[right] {$x_3$};
        	\fill (-3,2.3) circle (2pt) node[right] {$x_8$};
        \end{tikzpicture}
        \hspace{33mm}
        \begin{tikzpicture}
        	\draw[dashed] (0,0) -- (0,3.75);	
        	\draw[dashed] (-1,0.75) -- (-1,2.25);
        	\draw[dashed] (-3,-0.2) -- (-3,4);
        	\fill (0,0.75) circle (2pt) node[right] {$x'_6$};
        	\fill (0,2.25) circle (2pt) node[right] {$x'_5$};
        	\fill (0,3) circle (2pt) node[right] {$x'_2$};

        	\fill (-1,1.5) circle (2pt) node[right] {$x'_7$};

        	\fill (-3,0.375) circle (2pt) node[right] {$x'_3$};
        	\fill (-3,1.5) circle (2pt) node[right] {$x'_1$};
        	\fill (-3,2.625) circle (2pt) node[right] {$x'_8$};
        	\fill (-3,3.375) circle (2pt) node[right] {$x'_4$};
        \end{tikzpicture}
    \end{center}
        \caption{\label{fig:figure1}Two point configurations for $d=2$, $n=8$}
        \end{figure}
    correspond to the permutation $\sigma=38147652$.
    The combinatorial data for this configuration are
    $(\sigma,\ii)=(3\less{2}8\less{1}1\less{2}4\less{1}7\less{1}6\less{2}5\less{2}2)$.
    For the point configuration on the right we get the permutation $\sigma'=31847652$
    and the combinatorial data 
    $(\sigma',\ii')=(3\less{2}1\less{2}8\less{2}4\less{1}7\less{1}6\less{2}5\less{2}2)$.
\end{example}

All the points in $W_n^{\oplus d}$ with the same combinatorial data $(\sigma,\ii)$ make up a \emph{stratum} 
that we denote by $C(\sigma,\ii)$. This stratum is the relative interior of a polyhedral cone
of codimension $(i_1-1)+\dots+(i_{n-1}-1)$, that is, 
of dimension $(d+1)(n-1)-(i_1+\dots+i_{n-1})$ in~$W_n^{\oplus d}$.  

\begin{example}\label{ex:8points_permutation2}
The stratum $C(\sigma,\ii)=(3\less{2}8\less{1}1\less{2}4\less{1}7\less{1}6\less{2}5\less{2}2)$ of Example~\ref{ex:8points_permutation}
has dimension~$10$. It~is%
\begin{eqnarray*}
\left\{
\begin{pmatrix} 
x_{11} & x_{12} & \ldots & x_{18} \\
x_{21} & x_{22} & \ldots & x_{28}
\end{pmatrix}\right. \in W_8^{\oplus 2}\subset\R^{2\times 8} : 
\left.\begin{array}{c@{\,\,}c@{\,\,}c@{\,\,}c@{\,\,}c@{\,\,}c@{\,\,}c@{\,\,}c@{\,\,}c@{\,\,}c@{\,\,}c@{\,\,}c@{\,\,}c@{\,\,}c@{\,\,}c} 
    x_{13} & = & x_{18} & < & x_{11} & = & x_{14} & < & x_{17} & < & x_{16} & = & x_{15} & = & x_{12} \\
    x_{23} & < & x_{28} &   & x_{21} & < & x_{24} &   & x_{27} &   & x_{26} & < & x_{25} & < & x_{22}   
\end{array}
\right\}.
\end{eqnarray*}  
\end{example}
  
\begin{lemma}\label{lemma:combstrat}
The closure of each stratum $C(\sigma,\ii)\subset W_n^{\oplus d}$  
is a union of strata, namely of all strata $C(\sigma',\ii')$ such that 
\begin{itemize}
\item[$(*)$] 
        If $\ \ldots\sigma_k   \ldots\less{i_p'}\ldots\sigma_\ell\ldots\ $ appear in this order in $(\sigma',\ii')$, then\\  
    either $\ \ldots\sigma_k   \ldots\less{i_p} \ldots\sigma_\ell\ldots\ $ appear in this order in $(\sigma,\ii)$ with $i_p\le i_p'$,\\
    or     $\ \ldots\sigma_\ell\ldots\less{i_p} \ldots\sigma_k   \ldots\ $ appear in this order in $(\sigma,\ii)$ with $i_p  < i_p'$.
\end{itemize}
\end{lemma}

\begin{proof}
    Let $(x_1,\dots,x_n)$ be a configuration of points that lies in the stratum
    $C(\sigma,\ii)$, which is a relatively open polyhedral cone characterized by equations 
    and strict inequalities. The closure of the cone is given by the condition
    that the equations still hold, and the inequalities still hold weakly.
    Thus the smallest index of a coordinate in which two points $x_{\sigma_k}$ and $x_{\sigma_\ell}$
    differ may not go down when moving to $x_{\sigma_k}'$ and $x_{\sigma_\ell}'$,
    and if it stays the same, then the order in this coordinate must be preserved.
    
    Thus if $\ldots\sigma_k \ldots\less{i_p'}\ldots\sigma_\ell\ldots$ appear in this order in $(\sigma',\ii')$
    and $i_p'$ is not the smallest such index, then no condition is posed. If $i_p'$ is the 
    smallest such index, then we either have the same in $(\sigma,\ii)$,
    or there is a smaller index $i_p$. If there is no smaller index, then we still require
    that $\ldots\sigma_k \ldots\less{i_p'}\ldots\sigma_\ell\ldots$ appear in this order in $(\sigma,\ii)$.
    If there is a smaller index $i_p<i_p'$ between $\sigma_k$ and $\sigma_\ell$, then the order is arbitrary. 
        
    Condition $(*)$ describes this.   
\end{proof}
 
\begin{example}
    For the point configurations of Example~\ref{ex:8points_permutation}, 
    the $9$-dimensional stratum $C(\sigma',\ii')$ lies in the boundary of the $10$-dimensional stratum $C(\sigma,\ii)$.
\end{example}

\begin{definition}
    For $d\ge1$, $n\ge2$ the set of pairs $(\sigma,\ii)$ with the partial order
    $(\sigma,\ii)\ge(\sigma',\ii')$ described by Condition $(*)$ in Lemma~\ref{lemma:combstrat}
    is denoted by $\PosetOfStratification (d,n)$.
    Its minimal element $(1\less{{d+1}}2\less{{d+1}}\cdots\less{{d+1}}n)$ is denoted by~$\hat0$.
\end{definition}

\begin{example}\label{ex:Face23} The face poset of $\PosetOfStratification(2,3)$ 
    is displayed in Figure~\ref{fig:bigposet23}. (Compare  
    \cite[Fig.~2.3]{BjornerZiegler-combstrat}.)
    Here black dots denote cells in the complement of the arrangement given by the codimension~$2$ subspaces
    ``$x_i=x_j$'' (that is, lying in the configuration space
    $F(\R^2,3)$), while the white dots correspond to cells on the arrangement.
  
\begin{sidewaysfigure} 
     \begin{tikzpicture}[scale = 1]
    	\def\y{2.8}
    	\def \o{1.636363636363636363636}
    	\def \p{2.290909090909090909090}
              \def \coord{5.2727272727272727272727272727}

    \tikzstyle{block} = [rectangle, draw,   text width=5em, text centered, rounded corners, minimum height=4em]

     \draw (2.5,2) node[rectangle callout, draw, rounded corners, text  = black, minimum height=2cm, minimum width = 1.5cm,callout relative  pointer={(0.4,-0.3)}] {};
    \draw[densely dashed] (2.4,1.5) -- (2.4,2.8);
    \fill (2.4,2.5) circle (2pt) node[below right] {$x_2$};
    \draw[densely dashed] (2.1,1.2) -- (2.1,2.8);
    \fill (2.1,1.5 ) circle (2pt) node[below right] {$x_1$};
    \draw[densely dashed] (2.7,1.2) -- (2.7,2.8);
    \fill (2.7,1.8) circle (2pt) node[below right] {$x_3$};

     \draw (-0.75,-0.75) node[rectangle callout, draw, rounded corners, text  = black, minimum height=2cm, minimum width = 1.5cm,callout relative  pointer={(0.4,-0.3)}] {};
    \draw[densely dashed] (-0.6,0.0) -- (-0.6,-1.6);
    \fill (-0.6,-0.7) circle (2pt) node[below right] {$x_2$};
    \fill (-0.6,-0.2) circle (2pt) node[below right] {$x_3$};
    \draw[densely dashed] (-1.1,0.0) -- (-1.1,-1.6);
    \fill (-1.1,-1.1) circle (2pt) node[below right] {$x_1$};
    
     \draw (-2.85,-4.15) node[rectangle callout, draw, rounded corners, text  = black, minimum height=2cm, minimum width = 1.5cm,callout relative  pointer={(0.4,-0.3)}] {};
    \draw[densely dashed] (-2.75,-4.9) -- (-2.75,-3.3);
    \fill (-2.80,-3.8) circle (2pt) node[below right] {$x_3$};
    \fill (-2.70,-3.8) circle (2pt) node[below left] {$x_2$};
    \draw[densely dashed] (-3.3,-4.9) -- (-3.3,-3.3);
    \fill (-3.3,-4.4) circle (2pt) node[below right] {$x_1$};

     \draw (-1,-7) node[rectangle callout, draw, rounded corners, text  = black, minimum height=2cm, minimum width = 1.5cm,callout relative  pointer={(0.4,-.3)}] {};
    \draw[densely dashed] (-1.05,-6.2) -- (-1.05,-7.7);
    \fill (-1.05,-7.2) circle (2pt) node[below right ] {$x_1$};
    \fill (-1.0,-6.7) circle (2pt) node[below left] {$x_2$};
    \fill (-1.1,-6.7) circle (2pt) node[below right] {$x_3$};

     \draw (7,-12) node[rectangle callout, draw, rounded corners, text  = black, minimum height=2cm, minimum width = 1.5cm,callout relative  pointer={(0.4,.3)}] {};
    \draw[densely dashed] (6.975,-11.5) -- (6.975,-12.5);
    \fill (7.025,-12) circle (2pt) node[below left] {$x_2$};
    \fill (6.925,-12) circle (2pt) node[below right] {$x_3$};
    \fill (6.975,-11.93) circle (2pt) node[above right] {$x_1$};


    	\draw (\coord,-5+2*\y) -- (2,-5+\y);
    	\draw (\coord,-5+2*\y) -- (2+\o,-5+\y);
     	\draw (\coord,-5+2*\y) -- (20,-5+\y);
    	\draw (\coord,-5+2*\y) -- (20-\o,-5+\y);

    	\draw (\coord+\p, -5+2*\y) -- (2,-5+\y);
    	\draw (\coord+\p, -5+2*\y) -- (2+\o,-5+\y);
    	\draw (\coord+\p, -5+2*\y) -- (2+2*\o,-5+\y);
    	\draw (\coord+\p, -5+2*\y) -- (2+3*\o,-5+\y);

    	\draw (\coord+2*\p, -5+2*\y) -- (2+2*\o,-5+\y);
    	\draw (\coord+2*\p, -5+2*\y) -- (2+3*\o,-5+\y);
    	\draw (\coord+2*\p, -5+2*\y) -- (2+4*\o,-5+\y);
    	\draw (\coord+2*\p, -5+2*\y) -- (2+5*\o,-5+\y);

    	\draw (\coord+3*\p, -5+2*\y) -- (2+4*\o,-5+\y);
    	\draw (\coord+3*\p, -5+2*\y) -- (2+5*\o,-5+\y);
    	\draw (\coord+3*\p, -5+2*\y) -- (2+6*\o,-5+\y);
    	\draw (\coord+3*\p, -5+2*\y) -- (2+7*\o,-5+\y);

    	\draw (\coord+4*\p, -5+2*\y) -- (2+6*\o,-5+\y);
    	\draw (\coord+4*\p, -5+2*\y) -- (2+7*\o,-5+\y);
    	\draw (\coord+4*\p, -5+2*\y) -- (2+8*\o,-5+\y);
    	\draw (\coord+4*\p, -5+2*\y) -- (2+9*\o,-5+\y);

    	\draw (\coord+5*\p, -5+2*\y) -- (2+8*\o,-5+\y);
    	\draw (\coord+5*\p, -5+2*\y) -- (2+9*\o,-5+\y);
    	\draw (\coord+5*\p, -5+2*\y) -- (2+10*\o,-5+\y);
    	\draw (\coord+5*\p, -5+2*\y) -- (2+11*\o,-5+\y);


    	\draw (2,-5+\y) -- (0,-5);
    	\draw (2,-5+\y) -- (4,-5);
    	\draw (2,-5+\y) -- (12,-5);
    	\draw (2,-5+\y) -- (14,-5);

    	\draw (2+\o,-5+\y) -- (0,-5);
    	\draw (2+\o,-5+\y) -- (6,-5);
    	\draw (2+\o,-5+\y) -- (8,-5);
    	\draw (2+\o,-5+\y) -- (10,-5);

    	\draw (2+2*\o,-5+\y) -- (4,-5);
    	\draw (2+2*\o,-5+\y) -- (6,-5);
    	\draw (2+2*\o,-5+\y) -- (14,-5);
    	\draw (2+2*\o,-5+\y) -- (16,-5);

    	\draw (2+3*\o,-5+\y) -- (8,-5);
    	\draw (2+3*\o,-5+\y) -- (10,-5);
    	\draw (2+3*\o,-5+\y) -- (12,-5);
    	\draw (2+3*\o,-5+\y) -- (16,-5);

    	\draw (2+4*\o,-5+\y) -- (4,-5);
    	\draw (2+4*\o,-5+\y) -- (6,-5);
    	\draw (2+4*\o,-5+\y) -- (8,-5);
    	\draw (2+4*\o,-5+\y) -- (22,-5);

    	\draw (2+5*\o,-5+\y) -- (10,-5);
    	\draw (2+5*\o,-5+\y) -- (12,-5);
    	\draw (2+5*\o,-5+\y) -- (14,-5);
    	\draw (2+5*\o,-5+\y) -- (22,-5);

    	\draw (2+6*\o,-5+\y) -- (2,-5);
    	\draw (2+6*\o,-5+\y) -- (4,-5);
    	\draw (2+6*\o,-5+\y) -- (12,-5);
    	\draw (2+6*\o,-5+\y) -- (14,-5);

    	\draw (2+7*\o,-5+\y) -- (2,-5);
    	\draw (2+7*\o,-5+\y) -- (6,-5);
    	\draw (2+7*\o,-5+\y) -- (8,-5);
    	\draw (2+7*\o,-5+\y) -- (10,-5);

    	\draw (2+8*\o,-5+\y) -- (4,-5);
    	\draw (2+8*\o,-5+\y) -- (6,-5);
    	\draw (2+8*\o,-5+\y) -- (14,-5);
    	\draw (2+8*\o,-5+\y) -- (18,-5);

    	\draw (2+9*\o,-5+\y) -- (8,-5);
    	\draw (2+9*\o,-5+\y) -- (10,-5);
    	\draw (2+9*\o,-5+\y) -- (12,-5);
    	\draw (2+9*\o,-5+\y) -- (18,-5);

    	\draw (2+10*\o,-5+\y) -- (4,-5);
    	\draw (2+10*\o,-5+\y) -- (6,-5);
    	\draw (2+10*\o,-5+\y) -- (8,-5);
    	\draw (2+10*\o,-5+\y) -- (20,-5);

    	\draw (2+11*\o,-5+\y) -- (10,-5);
    	\draw (2+11*\o,-5+\y) -- (12,-5);
    	\draw (2+11*\o,-5+\y) -- (14,-5);
    	\draw (2+11*\o,-5+\y) -- (20,-5);


    	\draw(0,-5) -- (2,-5-\y);
    	\draw(0,-5) -- (4,-5-\y);

    	\draw(2,-5) -- (2,-5-\y);
    	\draw(2,-5) -- (4,-5-\y);

    	\draw(4,-5) -- (2,-5-\y);
    	\draw(4,-5) -- (18,-5-\y); 

    	\draw(6,-5) -- (2,-5-\y);
    	\draw(6,-5) -- (14,-5-\y); 

    	\draw(8,-5) -- (14,-5-\y);
    	\draw(8,-5) -- (20,-5-\y);

    	\draw(10,-5) -- (4,-5-\y);
    	\draw(10,-5) -- (20,-5-\y);

    	\draw(12,-5) -- (4,-5-\y);
    	\draw(12,-5) -- (16,-5-\y);

    	\draw(14,-5) -- (16,-5-\y);
    	\draw(14,-5) -- (18,-5-\y);

    	\draw(16,-5) -- (14,-5-\y);
    	\draw(16,-5) -- (16,-5-\y);

    	\draw(18,-5) -- (14,-5-\y);
    	\draw(18,-5) -- (16,-5-\y);

    	\draw(20,-5) -- (18,-5-\y);
    	\draw(20,-5) -- (20,-5-\y);

    	\draw(22,-5) -- (18,-5-\y);
    	\draw(22,-5) -- (20,-5-\y);


    	\draw (2,-5-\y) -- (11, -5-2*\y);
    	\draw (4,-5-\y) -- (11, -5-2*\y);
    	\draw (14,-5-\y) -- (11, -5-2*\y);
    	\draw (16,-5-\y) -- (11, -5-2*\y);
    	\draw (18,-5-\y) -- (11, -5-2*\y);
    	\draw (20,-5-\y) -- (11, -5-2*\y);


    \foreach \x in {0,2}{
              		\fill[color = white] (\x,-5) circle(3pt);
              		\draw (\x,-5) circle(3pt);
        };

    	\foreach \x in {4,6,...,14}
              		\fill(\x,-5) circle(3pt);

    	\foreach \x in {16,18,20,22}{
                         \fill[color = white] (\x,-5) circle(3pt);
              		\draw (\x,-5) circle(3pt);
    	};

    	\foreach \x in {2,3.6363636363636363636363,5.27272727272727272727272727272727,...,20} 
              		\fill (\x, -5+\y) circle(3pt);

    	\foreach \x in {5.2727272727272727272727272727272727,7.563636363636362,...,16.72727272727272727272727272 } 
              		\fill (\x,-5+2*\y) circle(3pt);

    	\foreach \x in {2,4}{
                         \fill[color = white] (\x,-5-\y) circle(3pt);
              		\draw (\x,-5-\y) circle(3pt);
    	};

    	\foreach \x in {14,16,18,20}{
              		\fill[color = white] (\x,-5-\y) circle(3pt);
              		\draw (\x,-5-\y) circle(3pt);
    	};

    	\fill[color = white] (11,-5-2*\y) circle (3pt);
    	\draw (11,-5-2*\y) circle (3pt);


    	\draw (\coord, -5+2*\y) node[above left]{$1\less{1}2\less{1}3$};

    	\draw (\coord+\p, -5+2*\y) node[above left]{$1\less{1}3\less{1}2$};
    	\draw (\coord+2*\p, -5+2*\y) node[above left]{$3\less{1}1\less{1}2$};
    	\draw (\coord+3*\p, -5+2*\y) node[above left]{$3\less{1}2\less{1}1$};
    	\draw (\coord+4*\p, -5+2*\y) node[above left]{$2\less{1}3\less{1}1$};
    	\draw (\coord+5*\p, -5+2*\y) node[above left]{$2\less{1}1\less{1}3$};

    	\draw (2, -5+\y)       node[above left]{$1\less{1}2\less{2}3$};
    	\draw (2+\o, -5+\y)    node[above left]{$1\less{1}3\less{2}2$};
    	\draw (2+2*\o, -5+\y)  node[above left]{$1\less{2}3\less{1}2$};
    	\draw (2+3*\o, -5+\y)  node[above left]{$3\less{2}1\less{1}2$};
    	\draw (2+4*\o, -5+\y)  node[above left]{$3\less{1}1\less{2}2$};
    	\draw (2+5*\o, -5+\y)  node[above left]{$3\less{1}2\less{2}1$};
    	\draw (2+6*\o, -5+\y)  node[above left]{$2\less{2}3\less{1}1$};
    	\draw (2+7*\o, -5+\y)  node[above left]{$3\less{2}2\less{1}1$};
    	\draw (2+8*\o, -5+\y)  node[above left]{$2\less{1}1\less{2}3$};
    	\draw (2+9*\o, -5+\y)  node[above left]{$2\less{1}3\less{2}1$};
    	\draw (2+10*\o, -5+\y) node[above left]{$1\less{2}2\less{1}3$};
    	\draw (2+11*\o, -5+\y) node[above left]{$2\less{2}1\less{1}3$};

    	\draw (0, -5) node[below left]{$1\less{1}2\less{3}3$};
    	\draw (2, -5) node[below left]{$2\less{3}3\less{1}1$};

    	\draw (4, -5) node[below left]{$1\less{2}2\less{2}3$};
    	\draw (6, -5) node[below left]{$1\less{2}3\less{2}2$};
    	\draw (8, -5) node[below left]{$3\less{2}1\less{2}2$};
    	\draw (10, -5) node[below left]{$3\less{2}2\less{2}1$};
    	\draw (12, -5) node[below left]{$2\less{2}3\less{2}1$};
    	\draw (14, -5) node[below left]{$2\less{2}1\less{2}3$};

    	\draw (16, -5) node[below left]{$1\less{3}3\less{1}2$};
    	\draw (18, -5) node[below left]{$2\less{1}1\less{3}3$};
    	\draw (20, -5) node[below left]{$1\less{3}2\less{1}3$};
    	\draw (22, -5) node[below left]{$3\less{1}1\less{3}2$};

    	\draw (2, -5-\y)  node[below left]{$1\less{2}2\less{3}3$};
    	\draw (4, -5-\y)  node[below left]{$2\less{3}3\less{2}1$};

    	\draw (14, -5-\y) node[below left]{$1\less{3}3\less{2}2$};
    	\draw (16, -5-\y) node[below left]{$2\less{2}1\less{3}3$};
    	\draw (18, -5-\y) node[below left]{$1\less{3}2\less{2}3$};
    	\draw (20, -5-\y) node[below left]{$3\less{2}1\less{3}2$};

    	\draw (11,-5-2*\y) node[below left]{$\hat0=1\less{3}2\less{3}3$};
    \end{tikzpicture} 
\caption{\label{fig:bigposet23}The face poset $\PosetOfStratification(2,3)$}
\end{sidewaysfigure}

\end{example}
By Lemma~\ref{lemma:combstrat}, the intersection of the strata of the
Fox--Neuwirth stratification of $W_n^{\oplus d}$ with the unit sphere $S(W_n^{\oplus d})$ yields
a regular CW decomposition of this sphere of dimension $d(n-1)-1$, whose face poset
is given by $\PosetOfStratification(d,n)$.
(Here the stratum $\{0\}$ corresponds to $\hat0\in\PosetOfStratification(d,n)$; it corresponds
to the empty cell in the cellular sphere.)

The union of the arrangement is composed of all the strata with at least one entry $\less{d+1}$,
while the strata where all comparisons $\less{k}$ have $k\le d$ lie in the complement.
In particular, the link of the arrangement (its intersection with the unit sphere)
is given by a subcomplex of the cell decomposition of the unit sphere
in $W_n^{\oplus d}$ induced by the stratifiction.

\begin{example}\label{ex:cellBUsphere}
    For $d\ge1$, $n=2$  the CW decomposition of the unit sphere in $W_2^{\oplus d}\cong\R^d$
    we obtain is the minimal, centrally-symmetric cell decomposition of
    the $(d-1)$-sphere that has two vertices, two edges, two $2$-cells,
    etc.: The two $i$-cells are $C(1\less{{d-i}}2)\cap S(W_2^{\oplus d})$ and~$C(2\less{{d-i}}1)\cap S(W_2^{\oplus d})$, for $0\le i< d$.
    Both $i$-cells lie in the boundary of both $j$-cells, for $0\le i<j< d$.
\end{example}

\begin{example}
    For $d=1$, $n\ge2$ the stratification of $W_n$ is given by the real arrangement of 
    $\binom n2$ hyperplanes $x_i=x_j$ in~$W_n$. The CW decomposition of the unit sphere in
    $W_n$ has $2^n-2$ vertices, indexed by
    $(\sigma_1\less{2}\cdots\less{2}\sigma_{k-1}\less{1}\sigma_k\less{2}\cdots\less{2}\sigma_n)$ with
    $1\le k<n$, $\sigma_1<\dots<\sigma_{k-1}$ and $\sigma_k<\dots<\sigma_n$,
    and $n!$ facets corresponding to the regions of the hyperplane arrangement, 
    indexed by $(\sigma_1\less{1}\cdots\less{1}\sigma_n)$ for permutations~$\sigma\in\Symm_n$.
\end{example}

We are dealing with
a partition (“stratification”) of a Euclidean space $\R^m$ into a finite number of relative-open polyhedral cones  
such that the relative boundary of each cone $C$ (i.e.\ its boundary in the affine hull) is a union of 
(necessarily: finitely many, lower-dimensional) cones of the stratification.
Furthermore the stratification is \emph{weakly essential} in the following sense:
some of the strata may not be pointed (i.e., have $\{0\}$ as a face of the closure), but $\{0\}$ is a stratum that lies in the
boundary of all other strata, which thus are proper subsets of their affine hulls.
In this situation any selection of points $v_C$ in the relative interiors of the strata~$C$
induces a PL homeomorphism between the order complex of the face poset of the stratification 
to a star-shaped PL ball that is a neighborhood of the origin in~$\R^m$. 

\begin{example}
    Figure~\ref{fig:figure3} shows a weakly essential stratification of $\R^2$, its face poset, 
and a realization of the order complex of its face poset.
\begin{figure}[ht!]
\begin{center} 
         \medskip
         \begin{tikzpicture}
         	\coordinate (A) at (0,0);
         	\coordinate (B) at (2,0);
         	\coordinate (C) at (3,0.5);
         	\coordinate (D) at (3,-0.5);
         	\coordinate (E) at (-0.5,-1.5);
         	\coordinate (F) at (-0.5,1.5);
         	\coordinate (G) at (1.5,0.75);
         	\coordinate (H) at (2,-1);
         	\coordinate (I) at (-1,0.5);
         	\draw (-3,1.5) -- (4,-2);
         	\draw (A) -- (3,1.5);
         	\draw (A) -- (4,0);
              	\draw (A) node[below left] {$0$};
              	\draw (F) node[above] {$C$};
              	\draw (G) node [above] {$C'$};
         \end{tikzpicture}
 \qquad
 \begin{tikzpicture}
 	\coordinate (A) at (0,0);
 	\coordinate (B) at (2,0);
 	\coordinate (C) at (3,0.5);
 	\coordinate (D) at (3,-0.5);
 	\coordinate (E) at (-0.5,-1.5);
 	\coordinate (F) at (-0.5,1.5);
 	\coordinate (G) at (1.5,0.75);
 	\coordinate (H) at (2,-1);
 	\coordinate (I) at (-1,0.5);

 \filldraw[fill = black!25] (A) -- (B) -- (C) -- (A);
 \filldraw[fill = black!25] (A) -- (C) -- (G) -- (A);
 \filldraw[fill = black!25] (A) -- (B) -- (D) -- (A);
 \filldraw[fill = black!25] (A) -- (E) -- (H) -- (A);
 \filldraw[fill = black!25] (A) -- (H) -- (D) -- (A);
 \filldraw[fill = black!25] (A) -- (G) -- (F) -- (A);
 \filldraw[fill = black!25] (A) -- (F) -- (I) -- (A);
 \filldraw[fill = black!25] (A) -- (I) -- (E) -- (A); 
	\draw (-3,1.5) -- (4,-2);
	\draw (A) -- (3,1.5);
	\draw (A) -- (4,0); 
 	\fill (A) circle (2pt) node[below left] {$0$};
 	\fill (B) circle (2pt);
 	\fill (C) circle (2pt);
 	\fill (D) circle (2pt);
 	\fill (E) circle (2pt);
 	\fill (F) circle (2pt) node[above] {$v(C)$};
 	\fill (H) circle (2pt);
 	\fill (G) circle (2pt) node [above]{$v(C')$};
 	\fill (I) circle (2pt);
 \end{tikzpicture}
         \\ \medskip
 \begin{tikzpicture}
          	\fill (0,0) circle (3pt);
          	\fill (1,0) circle (3pt);
          	\fill (2,0) circle (3pt);
          	\fill (3,0) circle (3pt);

          	\fill (0,-1) circle (3pt);
          	\fill (1,-1) circle (3pt);
          	\fill (2,-1) circle (3pt);
          	\fill (3,-1) circle (3pt);

          	\fill (1.5,-2) circle (3pt);
          	\draw (0,0) -- (0,-1) -- (1,0) --(1,-1) -- (2,0) -- (2,-1) -- (3,0)--(3,-1) -- (0,0);
          	\draw(1.5,-2) -- (0,-1);
          	\draw(1.5,-2) -- (1,-1);
          	\draw(1.5,-2) -- (2,-1);
          	\draw(1.5,-2) -- (3,-1);
          	\draw (1.5,-2) node[below]{$\{ 0 \}$};
          	\draw (3,0) node[above right] {$C$};
          	\draw (3,-1) node [below right] {$C'$};
 \end{tikzpicture}
          \medskip
    \end{center} 
    \caption{\label{fig:figure3}A stratification of $\R^2$, its face poset, and a realization of its order complex}
    \end{figure}
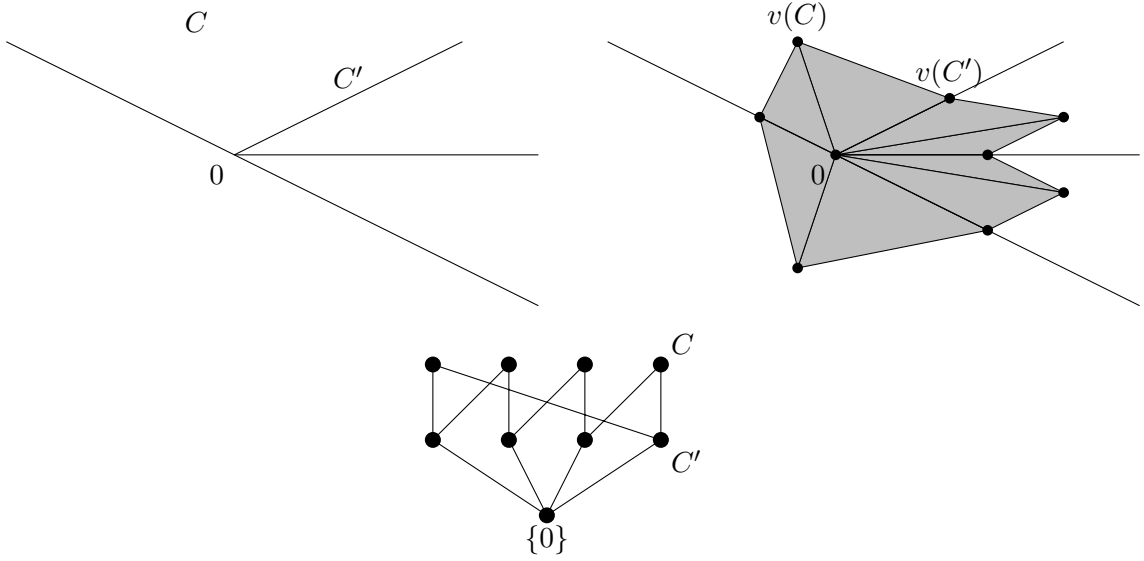
\end{example}
  
We now implement this construction for our specific stratification of $W_n^{\oplus d}$:
 
\begin{lemma}\label{lem:coordinates}
     For $d\ge1$, $n\ge2$, in each stratum $C(\sigma,\ii)$ a relative-interior point $v(\sigma,\ii)=(x_1,\dots,x_n)\in W_n^{\oplus d}$ 
     is obtained as follows: Set
     \[
     \tilde x_{\sigma(1)}:=0,\qquad \tilde x_{\sigma(j+1)}:= \tilde x_{\sigma(j)}+e_{i_j},
     \]
     where $e_i$ denotes the $i$th standard unit vector in $\R^d$, with $e_{d+1}:=0$.  
     The point $v(\sigma,\ii)$ is the image of this $\tilde x(\sigma,\ii)$
     under the orthogonal projection map $\nu$ from $\R^{d\times n}$ to the subspace $W_n^{\oplus d}$, which 
     translates the barycenter of the configuration to the origin, given by
     $x_j:=\nu(\tilde x_j)=\tilde x_j-\frac1n(\tilde x_1+\dots+\tilde x_n)$.
     
     These points $v(\sigma,\ii)$ 
     yield the vertices of a geometric realization of the order complex $\Delta\PosetOfStratification(d,n)$ by a
     star-shaped PL neighborhood $\SimplicialFullBall(d,n)$ of the origin in~$W_n^{\oplus d}$. Its boundary is a geometric realization
     of the barycentric subdivision of the cell decomposition $\CellComplexFullSphere(d,n)$ of $S(W_n^{\oplus d})$
     that is induced by the Fox--Neuwirth stratification.  
\end{lemma}

\begin{example}
    For the data $(\sigma,\ii)=(3\less{2}8\less{1}1\less{2}4\less{1}7\less{1}6\less{2}5\less{2}2)$ of the first 
    configuration in Example~\ref{ex:8points_permutation},  
    Lemma~\ref{lem:coordinates} yields the following configuration of Figure~\ref{fig:normalized28},
    which is then normalized by the map~$\nu$.
    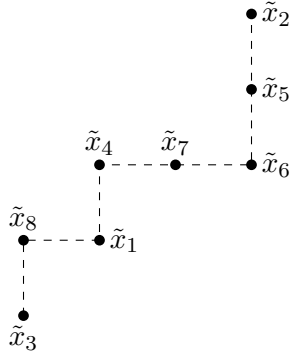
\begin{figure}[ht!]\begin{center}
            \begin{tikzpicture}[scale=1]
            	\draw[dashed] (0,0) -- (0,1) -- (1,1)--(1,2)--(2,2)--(3,2)--(3,3)--(3,4);	
    	        \fill (0,0) circle (2pt) node[below]{$\tilde x_3$};
            	\fill (0,1) circle (2pt) node[above] {$\tilde x_8$};
            	\fill (1,1) circle (2pt) node[right] {$\tilde x_1$};
            	\fill (1,2) circle (2pt) node[above] {$\tilde x_4$};
       	        \fill (2,2) circle (2pt) node[above] {$\tilde x_7$};
       	        \fill (3,2) circle (2pt) node[right] {$\tilde x_6$};
       	        \fill (3,3) circle (2pt) node[right] {$\tilde x_5$};
       	        \fill (3,4) circle (2pt) node[right] {$\tilde x_2$};
            \end{tikzpicture}
    \end{center}
    \caption{\label{fig:normalized28}The coordinates of Lemma~\ref{lem:coordinates} for the first configuration of
    Example~\ref{ex:8points_permutation}}
    \end{figure}
\end{example}

\begin{example}\label{ex:N13}
    For $n=3$, $d=1$ the star-shaped PL neighborhood $\SimplicialFullBall(1,3)$ of~$W_3\subset\R^3$ has
    $13$ vertices, some of whose coordinates according to Lemma~\ref{lem:coordinates} are 
    indicated in Figure~\ref{fig:figure4}.
\begin{figure}[ht!]
 \begin{center} 
     \begin{tikzpicture}[scale = 3.5,   decoration={markings,mark=at position 0.5 with {\arrow{triangle 60}}} , rotate = 30 ]
     \path (0,0) coordinate (origin);
     \filldraw[fill=black!25] 
        (  0:1cm)  coordinate (P0) node[anchor=west]        {~$(1\less{1}3\less{1}2) \longmapsto \nu(0,2,1)=(-1,1,0)\in W_3$} -- (30:.5) 
     -- ( 60:1cm)  coordinate (P1) node[anchor=south west]  { $(1\less{1}2\less{1}3) \longmapsto \nu(0,1,2)=(-1,0,1)\in W_3$} -- (90:.5) 
     -- (120:1cm)  coordinate (P2) -- (150:.5) 
     -- (180:1cm)  coordinate (P3) -- (210:.5)
     -- (240:1cm)  coordinate (P4) node[anchor= north]      { $(3\less{1}2\less{1}1) $} -- (270:.5) 
     -- (300:1cm)  coordinate (P5) node[anchor= west]       {~$(3\less{1}1\less{1}2) \longmapsto \nu(1,2,0)=(0,1,-1)\in W_3$} -- (330:.5) -- (0:1cm);
     
     \fill (P0) circle (.666pt);
     \fill (P1) circle (.666pt);
     \fill (P2) circle (.666pt);
     \fill (P3) circle (.666pt);
     \fill (P4) circle (.666pt);
     \fill (P5) circle (.666pt);
       \draw (0,0) -- ( 0:0.45cm);
       \draw (0:0.75cm) -- ( 0:1cm);
      \draw (0,0) -- (P1);
      \draw (0,0) -- (P2);
      \draw (0,0) -- (P3);
      \draw (0,0) -- (P4);
      \draw (0,0) -- (P5);
           
      \draw (0,0) -- ( 30:.5);
      \draw (0,0) -- ( 90:.5);
      \draw (0,0) -- (150:.5);
      \draw (0,0) -- (210:.5);  
      \draw (0,0) -- (270:.3);
      \draw (0,0) -- (330:.5);
     	
     \draw[dashed] (270:0.5) -- (270:1) node[right] {$x_1 = x_2$};
     \draw[dashed] ( 90:0.5) -- ( 90:1);
     \draw[dashed] (210:0.5) -- (210:1) node[left] {$x_2 = x_3$};
     \draw[dashed] ( 30:0.5) -- ( 30:1);
     \draw[dashed] (330:0.5) -- (330:1) node[right]{$x_1 = x_3$};
     \draw[dashed] (150:0.5) -- (150:1);   

            \fill[color = white] ( 30:.5) circle (.666pt);
            \fill[color = white] ( 90:.5) circle (.666pt);
            \fill[color = white] (150:.5) circle (.666pt);
            \fill[color = white] (210:.5) circle (.666pt);
            \fill[color = white] (270:.5) circle (.666pt);
            \fill[color = white] (330:.5) circle (.666pt);
            
            \draw ( 30:.5) circle (.666pt)node[anchor=north west] { $(1\less{1}2\less{2}3) \longmapsto \nu(0,1,1)=(-\frac23,\frac13,\frac13)\in W_3$};
            \draw ( 90:.5) circle (.666pt);
            \draw (150:.5) circle (.666pt);
            \draw (210:.5) circle (.666pt);
            \draw (270:.5) circle (.666pt)node[anchor=south ] { $(3\less{1}1\less{2}2)$};
            \draw (330:.5) circle (.666pt)node[anchor=north ] { $(1\less{2}3\less{1}2)$};
 
                \fill[color = white] (0,0) circle (.666pt);
                \draw (0,0) circle (.666pt);
     \end{tikzpicture}
 \end{center}
 \caption{\label{fig:figure4}Coordinates for the realization of $\SimplicialFullBall(1,3)$ in $W_3$ according to Lemma~\ref{lem:coordinates}}   
 \end{figure}
\end{example} 

In particular, the boundary of the star-shaped PL neighborhood $\SimplicialFullBall(d,n)$ is an $\Symm_n$-invariant PL sphere.
The link of the arrangement
is represented by an induced subcomplex of this sphere. 
The dual cell complex
of the sphere contains, as a subcomplex, a cellular model for the complement of the arrangement -- that is,
a simplicial complex that is a strong deformation retract of
the configuration space $F(\R^d,n)$. The cell complex is \emph{regular} in the sense that the
attaching maps of cells do not make identifications on the boundary; in particular, the barycentric subdivision
of any such complex is a simplicial complex, given by the order complex of the face poset
(see e.g.\ Björner et al.\ \cite[Sect.~4.7]{Z10-2}, Cooke \& Finney \cite{CookeFinney}, Munkres \cite{Munkres:AT}).

\begin{theorem}[Cell complex model for $F(\R^d,n)$]\label{thm:cell_complex}
    Let $d\ge1$ and $n\ge2$. Then $F(\R^d,n)$ contains, 
    as an equivariant strong deformation retract, a finite (and thus compact) regular CW complex $\CellComplexComplement(d,n)$
    of dimension $(d-1)(n-1)$ with $n!$ vertices, $n!$ facets, and $(n-1)n!$ ridges.
    
    The nonempty faces of $\CellComplexComplement(d,n)$ are indexed by the data of the
    form 
    \[
    (\sigma,\ii):=(\sigma_1\less{{i_1}}\sigma_2\less{{i_2}}\cdots\less{{i_{n-1}}}\sigma_n),
    \] 
    where $\sigma=\sigma_1\sigma_2\dots\sigma_n\in\Symm_n$ is a permutation,
    and $\ii=(i_1,\dots,i_{n-1})\in\{1,\dots,d\}^{n-1}$. Let $\PosetOfComplement(d,n)$
    be the set of these strings.
    
    The dimension of the cell $\check{c}(\sigma,\ii)$ associated with $(\sigma,\ii)$ 
    is $(i_1+\dots+i_{n-1})-(n-1)$.   
    
    The inclusion relation between cells $\check{c}(\sigma,\ii)$, and thus the partial order of the face poset
    $\PosetOfComplement(d,n)$, is as follows:  
    \[
     (\sigma,\ii)\preceq (\sigma',\ii')
    \]
    holds if and only if  
    \begin{itemize}
    \item[$(**)$] 
    whenever  $\ \ldots\sigma_k\ldots\less{{i_p'}}\ldots\sigma_\ell\ldots\ $ appear in this order in $(\sigma',\ii')$, then\\  
    either $\ \ldots\sigma_k   \ldots\less{{i_p }}\ldots\sigma_\ell\ldots\ $ appear in this order in $(\sigma,\ii)$ with $i_p\le i'_p$,\\
    or     $\ \ldots\sigma_\ell\ldots\less{{i_p }}\ldots\sigma_k\ldots\ $    appear in this order in $(\sigma,\ii)$ with $i_p < i'_p$.  
    \end{itemize} 
    The barycentric subdivision of the cell complex  $\CellComplexComplement(d,n)$ (that is, the order complex $\Delta\PosetOfComplement(d,n)$ 
    of its face poset) has a geometric realization in $W_n^{\oplus d}$ 
    with vertices $v(\sigma,\ii)=(x_1,\dots,x_n)\in W_n^{\oplus d}$ placed according to
    \[
     \tilde x_{\sigma(1)}:=0,\qquad \tilde x_{\sigma(j+1)}:= \tilde x_{\sigma(j)}+e_{i_j},
    \]
     where $e_i$ denotes the $i$th standard unit vector in $\R^d$, followed by
     orthogonal projection $\R^{d\times n}\rightarrow W_n^{\oplus d}$, given by
     $x_j:=\tilde x_j-\frac1n(\tilde x_1+\dots+\tilde x_n)$.
    This geometric realization in $W_n^{\oplus d}\subset\R^{d\times n}$ is an
    $\Symm_n$-equivariant strong deformation retract of~$F(\R^d,n)$.
    
    The group $\Symm_n$ acts on the poset $\PosetOfComplement(d,n)$ via $\pi\cdot (\sigma,\ii)= (\pi\sigma,\ii)$,
    and correspondingly by $\pi\cdot v(\sigma,\ii)= v(\pi\sigma,\ii)$ on the geometric realization 
    $\SimplicialComplement(d,n)$ with 
    vertex set $\{v(\sigma,\ii):(\sigma,\ii)\in\PosetOfComplement(d,n)\}$.
\end{theorem}

\begin{proof}
    Orthogonal projection $\Symm_n$-equivariantly retracts $F(\R^d,n)$ to $\overline F(\R^d,n)\subset W_n^{\oplus d}$,
    and further radial projection $\Symm_n$-equivariantly retracts this to a subset of 
    the boundary of the star-shaped neigborhood of the origin, $\partial\,\SimplicialFullBall(d,n)$,
    which is a simplicial realization of the cell decomposition of
    the $(d(n-1)-1)$-dimensional sphere $S(W_n^{\oplus d})$ given by the Fox--Neuwirth stratification.
    
    The same maps identify the link of the arrangement (that is, the intersection of its union with the 
	unit sphere) with the induced subcomplex of $\partial\,\SimplicialFullBall(d,n)$
    on the vertices $v(\sigma,\ii)$ that have some index $i_j=d+1$.

    As the cell decomposition of the $(d(n-1)-1)$-sphere $S(W_n^{\oplus d})$ is PL, 
    we can pass to the Poincaré--Alexander dual cell decomposition.  
    Thus for every vertex in the sphere $S(W_n^{\oplus d})$, there is a corresponding
    facet in the dual cell decomposition, which is given by its star in the
    barycentric subdivision of the cell decomposition. More generally, any nonempty cell $C(\sigma,\ii)$
    of dimension $(d+1)(n-1)-(i_1+\dots+i_n)-1$ has an associated dual cell
    $\check{c}(\sigma,\ii)$ of dimension $(d(n-1)-1)-((d+1)(n-1)-(i_1+\dots+i_{n-1})-1)=(i_1+\dots+i_{n-1})-(n-1)$,
    and the inclusion relation is just the opposite of the one in the primal complex,
    as described in Lemma~\ref{lemma:combstrat}.
    (Compare to Munkres \cite[\S\,64]{Munkres:AT}.)
    
    According to the Retraction Lemma \cite[Lemma~4.7.27]{Z10-2} \cite[Lemma~70.1]{Munkres:AT},
    the complement admits a strong deformation retraction to the subcomplex whose cells
    are indexed by all pairs $(\sigma,\ii)$ with all indices $i_j\neq d+1$.
    (This retraction is easy to describe in barycentric coordinates in the
    barycentric subdivision of the cell complex, that is, on the order complex of the face poset.
    The $\Symm_n$-action restricts to the subcomplex, and the retraction is canonical, and thus
    $\Symm_n$-equivariant.) 
    
    So we obtain a strong deformation retract of $F(\R^d,n)$ that is a cell complex with cells indexed
    by $(\sigma,\ii)$ with all indices $i_j\neq0$.
    The barycentric subdivision of this cell complex is realized as a simplicial complex in $W_n^{\oplus d}$
    as the induced subcomplex of $\partial N(d,n)$ on the vertices $v(\sigma,\ii)$ with all indices $i_j\neq d+1$.
    Thus also the $\Symm_n$-action restricts to this simplicial model for the complement.
\end{proof}

\begin{example}
    For $n=2$ the cell complex $\CellComplexComplement(d,2)$ turns out to be the centrally-symmetric  
    cell decomposition of $S^{d-1}$ whose $(i-1)$-cells are given by $\check{c}(1\less{i}2)$ and $\check{c}(2\less{i}1)$
    for $1\le i\le d$. (Compare to Example~\ref{ex:cellBUsphere}.)
\end{example}

\begin{example}
    For $d=1$ the configuration space $F(\R,n)$ is the complement of a real hyperplane arrangement in $W_n$.
    The cell complex $\CellComplexComplement(1,n)$ is $0$-dimensional, given by the $n!$ vertices
    $v(\sigma_1\less{1}\sigma_2\cdots\less{1}\sigma_n)\in W_n$, one in each chamber of the arrangmenent.
    (Compare to Example~\ref{ex:N13}.)
\end{example}

\begin{example}
    For $d=2$ the configuration space $F(\R^2,n)=F(\C,n)$ may be viewed as the complement of a 
    complex hyperplane arrangement, known as the braid arrangement.
    This is the particular situation studied by Arnol'd \cite{Arnold69}, Deligne \cite{Deligne72}, and many others.
    The cell complex that we obtain is a particularly interesting instance of the cell complex constructed
    by Salvetti \cite{Salvetti} for the complements of complexified hyperplane arrangements (see also \cite{BjornerZiegler-combstrat}).
    
    In this case we may simplify our notation a bit: The non-empty cells of the complex are indexed by
    $(\sigma_1\less{{i_1}}\cdots\less{{i_{n-1}}}\sigma_n)$ with $i_j\in\{1,2\}$. Thus we can identify our cells uniquely
    if we just write a vertical bar instead of “$\less{1}$”, no bar for “$\less{2}$”. The inclusion of a cell into the boundary
    of a higher-dimensional cell is then represented in the partial order by 
    (repeatedly) “removing a bar, and shuffling the two blocks that were separated by the bar”. 
    
    For example, for $n=3$ we obtain the face poset $\PosetOfComplement(2,3)$ displayed in 
    Figure~\ref{fig:complement_poset}. It is obtained by taking the elements
    of the partial order $\PosetOfStratification(2,3)$ displayed in Figure~\ref{fig:bigposet23} that are marked by black dots, 
    with labels simplified, and the partial order reversed.

\begin{figure}[ht!]
\begin{center}
    \medskip
\begin{tikzpicture}[scale=.8, rotate=180]
	\def \y{2.8}
	\def \o{1.636363636363636363636}
	\def \p{2.290909090909090909090}
    \def \coord{5.2727272727272727272727272727}

	\draw (\coord,-5+2*\y) -- (2,-5+\y);
	\draw (\coord,-5+2*\y) -- (2+\o,-5+\y);
 	\draw (\coord,-5+2*\y) -- (20,-5+\y);
	\draw (\coord,-5+2*\y) -- (20-\o,-5+\y);

	\draw (\coord+\p, -5+2*\y) -- (2,-5+\y);
	\draw (\coord+\p, -5+2*\y) -- (2+\o,-5+\y);
	\draw (\coord+\p, -5+2*\y) -- (2+2*\o,-5+\y);
	\draw (\coord+\p, -5+2*\y) -- (2+3*\o,-5+\y);

	\draw (\coord+2*\p, -5+2*\y) -- (2+2*\o,-5+\y);
	\draw (\coord+2*\p, -5+2*\y) -- (2+3*\o,-5+\y);
	\draw (\coord+2*\p, -5+2*\y) -- (2+4*\o,-5+\y);
	\draw (\coord+2*\p, -5+2*\y) -- (2+5*\o,-5+\y);

	\draw (\coord+3*\p, -5+2*\y) -- (2+4*\o,-5+\y);
	\draw (\coord+3*\p, -5+2*\y) -- (2+5*\o,-5+\y);
	\draw (\coord+3*\p, -5+2*\y) -- (2+6*\o,-5+\y);
	\draw (\coord+3*\p, -5+2*\y) -- (2+7*\o,-5+\y);

	\draw (\coord+4*\p, -5+2*\y) -- (2+6*\o,-5+\y);
	\draw (\coord+4*\p, -5+2*\y) -- (2+7*\o,-5+\y);
	\draw (\coord+4*\p, -5+2*\y) -- (2+8*\o,-5+\y);
	\draw (\coord+4*\p, -5+2*\y) -- (2+9*\o,-5+\y);

	\draw (\coord+5*\p, -5+2*\y) -- (2+8*\o,-5+\y);
	\draw (\coord+5*\p, -5+2*\y) -- (2+9*\o,-5+\y);
	\draw (\coord+5*\p, -5+2*\y) -- (2+10*\o,-5+\y);
	\draw (\coord+5*\p, -5+2*\y) -- (2+11*\o,-5+\y);

    \draw (2,-5+\y) -- (4,-5);
    \draw (2,-5+\y) -- (12,-5);
    \draw (2,-5+\y) -- (14,-5);

	\draw (2+\o,-5+\y) -- (6,-5);
	\draw (2+\o,-5+\y) -- (8,-5);
	\draw (2+\o,-5+\y) -- (10,-5);

	\draw (2+2*\o,-5+\y) -- (4,-5);
	\draw (2+2*\o,-5+\y) -- (6,-5);
	\draw (2+2*\o,-5+\y) -- (14,-5);

	\draw (2+3*\o,-5+\y) -- (8,-5);
	\draw (2+3*\o,-5+\y) -- (10,-5);
	\draw (2+3*\o,-5+\y) -- (12,-5);

	\draw (2+4*\o,-5+\y) -- (4,-5);
	\draw (2+4*\o,-5+\y) -- (6,-5);
	\draw (2+4*\o,-5+\y) -- (8,-5);

	\draw (2+5*\o,-5+\y) -- (10,-5);
	\draw (2+5*\o,-5+\y) -- (12,-5);
	\draw (2+5*\o,-5+\y) -- (14,-5);

	\draw (2+6*\o,-5+\y) -- (4,-5);
	\draw (2+6*\o,-5+\y) -- (12,-5);
	\draw (2+6*\o,-5+\y) -- (14,-5);

	\draw (2+7*\o,-5+\y) -- (6,-5);
	\draw (2+7*\o,-5+\y) -- (8,-5);
	\draw (2+7*\o,-5+\y) -- (10,-5);

	\draw (2+8*\o,-5+\y) -- (4,-5);
	\draw (2+8*\o,-5+\y) -- (6,-5);
	\draw (2+8*\o,-5+\y) -- (14,-5);

	\draw (2+9*\o,-5+\y) -- (8,-5);
	\draw (2+9*\o,-5+\y) -- (10,-5);
	\draw (2+9*\o,-5+\y) -- (12,-5);

	\draw (2+10*\o,-5+\y) -- (4,-5);
	\draw (2+10*\o,-5+\y) -- (6,-5);
	\draw (2+10*\o,-5+\y) -- (8,-5);

	\draw (2+11*\o,-5+\y) -- (10,-5);
	\draw (2+11*\o,-5+\y) -- (12,-5);
	\draw (2+11*\o,-5+\y) -- (14,-5);
 
	\foreach \x in {4,6,...,14}
          		\fill(\x,-5) circle(3pt);
 
	\foreach \x in {2,3.6363636363636363636363,5.27272727272727272727272727272727,...,20} 
          		\fill (\x, -5+\y) circle(3pt);

	\foreach \x in {5.2727272727272727272727272727272727,7.563636363636362,...,16.72727272727272727272727272 } 
          		\fill (\x,-5+2*\y) circle(3pt);
	\draw (\coord,      -5+2*\y) node[below]{$1|2|3$};
	\draw (\coord+\p,   -5+2*\y) node[below]{$1|3|2$};
	\draw (\coord+2*\p, -5+2*\y) node[below]{$3|1|2$};
	\draw (\coord+3*\p, -5+2*\y) node[below]{$3|2|1$};
	\draw (\coord+4*\p, -5+2*\y) node[below]{$2|3|1$};
	\draw (\coord+5*\p, -5+2*\y) node[below]{$2|1|3$};

	\draw (2, -5+\y)       node[below left ]{$1|2 3$};
	\draw (2+\o, -5+\y)    node[below left ]{$1|3 2$};
	\draw (2+2*\o, -5+\y)  node[below left ]{$1 3|2$};
	\draw (2+3*\o, -5+\y)  node[below left ]{$3 1|2$};
	\draw (2+4*\o, -5+\y)  node[below left ]{$3|1 2$};
	\draw (2+5*\o, -5+\y)  node[below left ]{$3|2 1$};
	\draw (2+6*\o, -5+\y)  node[below left ]{$2 3|1$};
	\draw (2+7*\o, -5+\y)  node[below left ]{$3 2|1$};
	\draw (2+8*\o, -5+\y)  node[below left ]{$2|1 3$};
	\draw (2+9*\o, -5+\y)  node[below left ]{$2|3 1$};
	\draw (2+10*\o, -5+\y) node[below left ]{$1 2|3$};
	\draw (2+11*\o, -5+\y) node[below left ]{$2 1|3$};
	\draw (4, -5) node[above ]{$1 2 3$};
	\draw (6, -5) node[above ]{$1 3 2$};
	\draw (8, -5) node[above ]{$3 1 2$};
	\draw (10, -5) node[above ]{$3 2 1$};
	\draw (12, -5) node[above ]{$2 3 1$};
	\draw (14, -5) node[above ]{$2 1 3$};
\end{tikzpicture}
\medskip
\end{center}
\caption{\label{fig:complement_poset}The poset $\PosetOfComplement(2,3)$}
\end{figure}

\noindent
We see that the cell complex $\CellComplexComplement(2,n)$ is particularly nice:
\begin{itemize}
    \item The cell complex $\CellComplexComplement(2,n)$ is a regular cell complex of dimension $n-1$.
    \item It has $n!$ facets, given by permutations $\sigma_1\sigma_2\ldots\sigma_n$,
    and $n!$ vertices, given by the barred permutations $\sigma_1|\sigma_2|\dots|\sigma_n$.
    \item It has $\binom{n-1}{i-1}n!$ cells of dimension $i$, given by permutations with $n-i-1$ bars.
    \item The order relation is generated by the operation of “remove a bar, and merge the two adjacent blocks separated
    by the bar by a shuffle”.
    \item The maximal cells have the combinatorial structure of an $(n-1)$-dimensional permutahedron.
    (Thus all cells have the combinatorial structure of simple polytopes, namely of products of permutahedra.)
\end{itemize}
Figure \ref{fig:cellsofcomplex23} indicates the structure of the cell complex $\CellComplexComplement(2,3)$:
This is a regular cell complex with $3!=6$ vertices, $2\cdot3!=12$ edges, and
$3!=6$ $2$-cells, which are hexagons. The figure displays the six $2$-cells shaded in separate drawings;
for example, the $2$-cell
$\check{c}(123)$ is bounded by the six edges 
$\check{c}(1|23),\check{c}(13|2),\check{c}(3|12),\check{c}(23|1),\check{c}(2|13),\check{c}(12|3)$.
In the complex $\CellComplexComplement(2,3)$ each edge is contained in three of the six hexagon $2$-cells.
The six edges in the boundary of each $2$-cell lie in two different orbits of the group $\Symm_3$;
for the $2$-cell $\check{c}(123)$ displayed in Figure~\ref{fig:cellsofcomplex23}, the three edges
$\check{c}(1|23),\check{c}(3|12),\check{c}(2|13)$ lie in one $\Symm_3$-orbit, while
$\check{c}(13|2),\check{c}(23|1),\check{c}(12|3)$ lie in the other one.

\begin{figure}[hp]
      \begin{tikzpicture}[scale = 3, rotate = 30  ]
        \path (0,0) coordinate (origin);
        \path (0:1cm) coordinate (P0);
        \path (1*60:1cm) coordinate (P1);
        \path (2*60:1cm) coordinate (P2);
        \path (3*60:1cm) coordinate (P3);
        \path (4*60:1cm) coordinate (P4);
        \path (5*60:1cm) coordinate (P5);

        \fill (P0) circle (0.5pt);
        \fill (P1) circle (0.5pt);
        \fill (P2) circle (0.5pt);
        \fill (P3) circle (0.5pt);
        \fill (P4) circle (0.5pt);
        \fill (P5) circle (0.5pt);
    \fill [color=black!30] (0:1) arc (240:180:1) -- (60:1) arc (300:240:1) -- (120:1) arc (360:300:1) -- (180:1) arc (420:360:1) -- (240:1) arc (480:420:1) -- (300:1) arc (540:480:1); 
    \draw [line width=1.3] (0:1) arc (240:180:1) -- (60:1) arc (300:240:1) -- (120:1) arc (360:300:1) -- (180:1) arc (420:360:1) -- (240:1) arc (480:420:1) -- (300:1) arc (540:480:1); 

        \foreach \x in{0,1,2,3,4,5}
        	\draw[ decoration={markings,mark=at position 0.5 with {\arrow[line width = 0.8mm]{stealth}}},postaction = {decorate}] (\x*60:1) arc (\x*60:60+\x*60:1) ;

            \foreach \x in{0,1,2,3,4,5}{
    	\draw[ color = black!30,decoration={markings,mark=at position 1 with {\arrow[line width = 0.8mm, color = black]{stealth}}},postaction = {decorate}] (\x*60:1) arc   (\x*60+240:210+\x*60:1) ;
          \draw (\x*60 :1) arc (\x*60+240:180+\x*60:1) ;
   };

        \draw (P1) node[above] {$1|2|3$};
        \draw (P2) node[left] {$2|1|3$};
        \draw (P3) node[left] {$2|3|1$};
        \draw (P4) node[below] {$3|2|1$};
        \draw (P5) node[right] {$3|1|2$};
        \draw (P0) node[right] {$1|3|2$};

        \draw (30:1cm) node[above right] {$1|32$};
        \draw[xshift = -14, yshift = -7.5] (30:1cm) node[above right] {$1|23$};

        \draw (90:1cm) node[above  left] {$21|3$};

        \draw[xshift = -6, yshift = -12] (90:1cm) node[above right] {$12|3$};

        \draw (150:1cm) node[above left] {$2|31$};
        \draw[xshift = 7, yshift = -4] (150:1cm) node[above right] {$2|13$};

        \draw (210:1cm) node[below left] {$32|1$};
        \draw[xshift = 4, yshift = 6] (210:1cm) node[above right] {$23|1$};

        \draw (270:1cm) node[below right] {$3|21$};
        \draw[xshift =- 4, yshift = 10] (270:1cm) node[above right] {$3|12$};

        \draw (330:1cm) node[ right] {$31|2$};
        \draw[xshift =- 15, yshift = 5] (330:1cm) node[above right] {$13|2$};

        \draw (0,0) node {$123$};
        \draw[decoration={markings,mark=at position1 with {\arrow[line width = 0.5mm, rotate = -6]{stealth}}}, postaction={decorate}] (180:0.13) arc  (180:450:0.13);

        \end{tikzpicture}
        \quad
      \begin{tikzpicture}[scale = 3, rotate = 30  ]
        \path (0,0) coordinate (origin);
        \path (0:1cm) coordinate (P0);
        \path (1*60:1cm) coordinate (P1);
        \path (2*60:1cm) coordinate (P2);
        \path (3*60:1cm) coordinate (P3);
        \path (4*60:1cm) coordinate (P4);
        \path (5*60:1cm) coordinate (P5);

        \fill (P0) circle (0.5pt);
        \fill (P1) circle (0.5pt);
        \fill (P2) circle (0.5pt);
        \fill (P3) circle (0.5pt);
        \fill (P4) circle (0.5pt);
        \fill (P5) circle (0.5pt);
    \fill [color=black!30] (0:1) arc (0:60:1) -- (60:1) arc (300:240:1) -- (120:1) arc (360:300:1) -- (180:1) arc (180:240:1) -- (240:1) arc (480:420:1) -- (300:1) arc (540:480:1); 
    \draw[line width = 1.3] (0:1) arc (0:60:1) -- (60:1) arc (300:240:1) -- (120:1) arc (360:300:1) -- (180:1) arc (180:240:1) -- (240:1) arc (480:420:1) -- (300:1) arc (540:480:1); 

        \foreach \x in{0,1,2,3,4,5}
        	\draw[ decoration={markings,mark=at position 0.5 with {\arrow[line width = 0.8mm]{stealth}}},postaction = {decorate}] (\x*60:1) arc (\x*60:60+\x*60:1) ;

          \foreach \x in{0,1,2,3,4,5}{
    	\draw[ color = black!30,decoration={markings,mark=at position 1 with {\arrow[line width = 0.8mm, color = black]{stealth}}},postaction = {decorate}] (\x*60:1) arc   (\x*60+240:210+\x*60:1) ;
          \draw (\x*60 :1) arc (\x*60+240:180+\x*60:1) ;
   };

        \draw (P1) node[above] {$1|2|3$};
        \draw (P2) node[left] {$2|1|3$};
        \draw (P3) node[left] {$2|3|1$};
        \draw (P4) node[below] {$3|2|1$};
        \draw (P5) node[right] {$3|1|2$};
        \draw (P0) node[right] {$1|3|2$};

        \draw (30:1cm) node[above right] {$1|32$};
        \draw[xshift = -14, yshift = -7.5] (30:1cm) node[above right] {$1|23$};

        \draw (90:1cm) node[above  left] {$21|3$};

        \draw[xshift = -6, yshift = -12] (90:1cm) node[above right] {$12|3$};

        \draw (150:1cm) node[above left] {$2|31$};
        \draw[xshift = 7, yshift = -4] (150:1cm) node[above right] {$2|13$};

        \draw (210:1cm) node[below left] {$32|1$};
        \draw[xshift = 4, yshift = 6] (210:1cm) node[above right] {$23|1$};

        \draw (270:1cm) node[below right] {$3|21$};
        \draw[xshift =- 4, yshift = 10] (270:1cm) node[above right] {$3|12$};

        \draw (330:1cm) node[ right] {$31|2$};
        \draw[xshift =- 15, yshift = 5] (330:1cm) node[above right] {$13|2$};

        \draw (0,0) node {$132$};
        \draw[decoration={markings,mark=at position1 with {\arrow[line width = 0.5mm, rotate = -6 ]{stealth}}}, postaction={decorate}] (180:0.13) arc  (180:450:0.13);
\end{tikzpicture} 
    
  \begin{tikzpicture}[scale = 3, rotate = 30  ]
    \path (0,0) coordinate (origin);
    \path (0:1cm) coordinate (P0);
    \path (1*60:1cm) coordinate (P1);
    \path (2*60:1cm) coordinate (P2);
    \path (3*60:1cm) coordinate (P3);
    \path (4*60:1cm) coordinate (P4);
    \path (5*60:1cm) coordinate (P5);

    \fill (P0) circle (0.5pt);
    \fill (P1) circle (0.5pt);
    \fill (P2) circle (0.5pt);
    \fill (P3) circle (0.5pt);
    \fill (P4) circle (0.5pt);
    \fill (P5) circle (0.5pt);
\fill [color=black!30] (0:1) arc (240:180:1) -- (60:1) arc (60:120:1) -- (120:1) arc (360:300:1) -- (180:1) arc (420:360:1) -- (240:1) arc (240:300:1) -- (300:1) arc (540:480:1); 
\draw[line width = 1.3] (0:1) arc (240:180:1) -- (60:1) arc (60:120:1) -- (120:1) arc (360:300:1) -- (180:1) arc (420:360:1) -- (240:1) arc (240:300:1) -- (300:1) arc (540:480:1); 
    \foreach \x in{0,1,2,3,4,5}
    	\draw[ decoration={markings,mark=at position 0.5 with {\arrow[line width = 0.8mm]{stealth}}},postaction = {decorate}] (\x*60:1) arc (\x*60:60+\x*60:1) ;

      \foreach \x in{0,1,2,3,4,5}{
    	\draw[ color = black!30,decoration={markings,mark=at position 1 with {\arrow[line width = 0.8mm, color = black]{stealth}}},postaction = {decorate}] (\x*60:1) arc   (\x*60+240:210+\x*60:1) ;
          \draw (\x*60 :1) arc (\x*60+240:180+\x*60:1) ;
   };

    \draw (P1) node[above] {$1|2|3$};
    \draw (P2) node[left] {$2|1|3$};
    \draw (P3) node[left] {$2|3|1$};
    \draw (P4) node[below] {$3|2|1$};
    \draw (P5) node[right] {$3|1|2$};
    \draw (P0) node[right] {$1|3|2$};

    \draw (30:1cm) node[above right] {$1|32$};
    \draw[xshift = -14, yshift = -7.5] (30:1cm) node[above right] {$1|23$};

    \draw (90:1cm) node[above  left] {$21|3$};

    \draw[xshift = -6, yshift = -12] (90:1cm) node[above right] {$12|3$};

    \draw (150:1cm) node[above left] {$2|31$};
    \draw[xshift = 7, yshift = -4] (150:1cm) node[above right] {$2|13$};

    \draw (210:1cm) node[below left] {$32|1$};
    \draw[xshift = 4, yshift = 6] (210:1cm) node[above right] {$23|1$};

    \draw (270:1cm) node[below right] {$3|21$};
    \draw[xshift =- 4, yshift = 10] (270:1cm) node[above right] {$3|12$};

    \draw (330:1cm) node[ right] {$31|2$};
    \draw[xshift =- 15, yshift = 5] (330:1cm) node[above right] {$13|2$};

    \draw (0,0) node {$213$};
    \draw[decoration={markings,mark=at position1 with {\arrow[line width = 0.5mm, rotate = -6]{stealth}}}, postaction={decorate}] (180:0.13) arc  (180:450:0.13);

    \end{tikzpicture}
    \quad
      \begin{tikzpicture}[scale = 3, rotate = 30  ]
        \path (0,0) coordinate (origin);
        \path (0:1cm) coordinate (P0);
        \path (1*60:1cm) coordinate (P1);
        \path (2*60:1cm) coordinate (P2);
        \path (3*60:1cm) coordinate (P3);
        \path (4*60:1cm) coordinate (P4);
        \path (5*60:1cm) coordinate (P5);

        \fill (P0) circle (0.5pt);
        \fill (P1) circle (0.5pt);
        \fill (P2) circle (0.5pt);
        \fill (P3) circle (0.5pt);
        \fill (P4) circle (0.5pt);
        \fill (P5) circle (0.5pt);
    \fill [color=black!30] (0:1) arc (240:180:1) -- (60:1) arc (60:120:1) -- (120:1) arc (120:180:1) -- (180:1) arc (420:360:1) -- (240:1) arc (240:300:1) -- (300:1) arc (300:360:1); 
    \draw[line width = 1.3] (0:1) arc (240:180:1) -- (60:1) arc (60:120:1) -- (120:1) arc (120:180:1) -- (180:1) arc (420:360:1) -- (240:1) arc (240:300:1) -- (300:1) arc (300:360:1); 
        \foreach \x in{0,1,2,3,4,5}
        	\draw[ decoration={markings,mark=at position 0.5 with {\arrow[line width = 0.8mm]{stealth}}},postaction = {decorate}] (\x*60:1) arc (\x*60:60+\x*60:1) ;

    \foreach \x in{0,1,2,3,4,5}{
    	\draw[ color = black!30,decoration={markings,mark=at position 1 with {\arrow[line width = 0.8mm, color = black]{stealth}}},postaction = {decorate}] (\x*60:1) arc   (\x*60+240:210+\x*60:1) ;
          \draw (\x*60 :1) arc (\x*60+240:180+\x*60:1) ;
   };

        \draw (P1) node[above] {$1|2|3$};
        \draw (P2) node[left] {$2|1|3$};
        \draw (P3) node[left] {$2|3|1$};
        \draw (P4) node[below] {$3|2|1$};
        \draw (P5) node[right] {$3|1|2$};
        \draw (P0) node[right] {$1|3|2$};

        \draw (30:1cm) node[above right] {$1|32$};
        \draw[xshift = -14, yshift = -7.5] (30:1cm) node[above right] {$1|23$};

        \draw (90:1cm) node[above  left] {$21|3$};

        \draw[xshift = -6, yshift = -12] (90:1cm) node[above right] {$12|3$};

        \draw (150:1cm) node[above left] {$2|31$};
        \draw[xshift = 7, yshift = -4] (150:1cm) node[above right] {$2|13$};

        \draw (210:1cm) node[below left] {$32|1$};
        \draw[xshift = 4, yshift = 6] (210:1cm) node[above right] {$23|1$};

        \draw (270:1cm) node[below right] {$3|21$};
        \draw[xshift =- 4, yshift = 10] (270:1cm) node[above right] {$3|12$};

        \draw (330:1cm) node[ right] {$31|2$};
        \draw[xshift =- 15, yshift = 5] (330:1cm) node[above right] {$13|2$};

        \draw (0,0) node {$231$};
        \draw[decoration={markings,mark=at position1 with {\arrow[line width = 0.5mm, rotate = -6]{stealth}}}, postaction={decorate}] (180:0.13) arc  (180:450:0.13);
        \end{tikzpicture}
    
          \begin{tikzpicture}[scale = 3.17, rotate = 30  ]
            \path (0,0) coordinate (origin);
            \path (0:1cm) coordinate (P0);
            \path (1*60:1cm) coordinate (P1);
            \path (2*60:1cm) coordinate (P2);
            \path (3*60:1cm) coordinate (P3);
            \path (4*60:1cm) coordinate (P4);
            \path (5*60:1cm) coordinate (P5);

            \fill (P0) circle (0.5pt);
            \fill (P1) circle (0.5pt);
            \fill (P2) circle (0.5pt);
            \fill (P3) circle (0.5pt);
            \fill (P4) circle (0.5pt);
            \fill (P5) circle (0.5pt);
        \fill [color=black!30] (0:1) arc (0:60:1) -- (60:1) arc (300:240:1) -- (120:1) arc (120:180:1) -- (180:1) arc (180:240:1) -- (240:1) arc (480:420:1) -- (300:1) arc (300:360:1); 
        \draw[line width = 1.3]  (0:1) arc (0:60:1) -- (60:1) arc (300:240:1) -- (120:1) arc (120:180:1) -- (180:1) arc (180:240:1) -- (240:1) arc (480:420:1) -- (300:1) arc (300:360:1); 

            \foreach \x in{0,1,2,3,4,5}
            	\draw[ decoration={markings,mark=at position 0.5 with {\arrow[line width = 0.8mm]{stealth}}},postaction = {decorate}] (\x*60:1) arc (\x*60:60+\x*60:1) ;

    \foreach \x in{0,1,2,3,4,5}{
    	\draw[ color = black!30,decoration={markings,mark=at position 1 with {\arrow[line width = 0.8mm, color = black]{stealth}}},postaction = {decorate}] (\x*60:1) arc   (\x*60+240:210+\x*60:1) ;
          \draw (\x*60 :1) arc (\x*60+240:180+\x*60:1) ;
   };

            \draw (P1) node[above] {$1|2|3$};
            \draw (P2) node[left] {$2|1|3$};
            \draw (P3) node[left] {$2|3|1$};
            \draw (P4) node[below] {$3|2|1$};
            \draw (P5) node[right] {$3|1|2$};
            \draw (P0) node[right] {$1|3|2$};

            \draw (30:1cm) node[above right] {$1|32$};
            \draw[xshift = -14, yshift = -7.5] (30:1cm) node[above right] {$1|23$};

            \draw (90:1cm) node[above  left] {$21|3$};

            \draw[xshift = -6, yshift = -12] (90:1cm) node[above right] {$12|3$};

            \draw (150:1cm) node[above left] {$2|31$};
            \draw[xshift = 7, yshift = -4] (150:1cm) node[above right] {$2|13$};

            \draw (210:1cm) node[below left] {$32|1$};
            \draw[xshift = 4, yshift = 6] (210:1cm) node[above right] {$23|1$};

            \draw (270:1cm) node[below right] {$3|21$};
            \draw[xshift =- 4, yshift = 10] (270:1cm) node[above right] {$3|12$};

            \draw (330:1cm) node[ right] {$31|2$};
            \draw[xshift =- 15, yshift = 5] (330:1cm) node[above right] {$13|2$};

            \draw (0,0) node {$312$};
            \draw[decoration={markings,mark=at position1 with {\arrow[line width = 0.5mm, rotate = -6]{stealth}}}, postaction={decorate}] (180:0.13) arc  (180:450:0.13);

            \end{tikzpicture}
        \quad
          \begin{tikzpicture}[scale = 3, rotate = 30  ]
            \path (0,0) coordinate (origin);
            \path (0:1cm) coordinate (P0);
            \path (1*60:1cm) coordinate (P1);
            \path (2*60:1cm) coordinate (P2);
            \path (3*60:1cm) coordinate (P3);
            \path (4*60:1cm) coordinate (P4);
            \path (5*60:1cm) coordinate (P5);

            \fill (P0) circle (0.5pt);
            \fill (P1) circle (0.5pt);
            \fill (P2) circle (0.5pt);
            \fill (P3) circle (0.5pt);
            \fill (P4) circle (0.5pt);
            \fill (P5) circle (0.5pt);
        \fill [color=black!30] (0:1) arc (0:60:1) -- (60:1) arc (60:120:1) -- (120:1) arc (120:180:1) -- (180:1) arc (180:240:1) -- (240:1) arc (240:300:1) -- (300:1) arc (300:360:1); 
        \draw[line width = 1.3]  (0:1) arc (0:60:1) -- (60:1) arc (60:120:1) -- (120:1) arc (120:180:1) -- (180:1) arc (180:240:1) -- (240:1) arc (240:300:1) -- (300:1) arc (300:360:1); 

            \foreach \x in{0,1,2,3,4,5}
            	\draw[ decoration={markings,mark=at position 0.5 with {\arrow[line width = 0.8mm]{stealth}}},postaction = {decorate}] (\x*60:1) arc (\x*60:60+\x*60:1) ;

    \foreach \x in{0,1,2,3,4,5}{
    	\draw[ color = black!30,decoration={markings,mark=at position 1 with {\arrow[line width = 0.8mm, color = black]{stealth}}},postaction = {decorate}] (\x*60:1) arc   (\x*60+240:210+\x*60:1) ;
          \draw (\x*60 :1) arc (\x*60+240:180+\x*60:1) ;
   };

            \draw (P1) node[above] {$1|2|3$};
            \draw (P2) node[left] {$2|1|3$};
            \draw (P3) node[left] {$2|3|1$};
            \draw (P4) node[below] {$3|2|1$};
            \draw (P5) node[right] {$3|1|2$};
            \draw (P0) node[right] {$1|3|2$};

            \draw (30:1cm) node[above right] {$1|32$};
            \draw[xshift = -14, yshift = -7.5] (30:1cm) node[above right] {$1|23$};

            \draw (90:1cm) node[above  left] {$21|3$};

            \draw[xshift = -6, yshift = -12] (90:1cm) node[above right] {$12|3$};

            \draw (150:1cm) node[above left] {$2|31$};
            \draw[xshift = 7, yshift = -4] (150:1cm) node[above right] {$2|13$};

            \draw (210:1cm) node[below left] {$32|1$};
            \draw[xshift = 4, yshift = 6] (210:1cm) node[above right] {$23|1$};

            \draw (270:1cm) node[below right] {$3|21$};
            \draw[xshift =- 4, yshift = 10] (270:1cm) node[above right] {$3|12$};

            \draw (330:1cm) node[ right] {$31|2$};
            \draw[xshift =- 15, yshift = 5] (330:1cm) node[above right] {$13|2$};

            \draw (0,0) node {$321$};
            \draw[decoration={markings,mark=at position1 with {\arrow[line width = 0.5mm, rotate = -6]{stealth}}}, postaction={decorate}] (180:0.13) arc  (180:450:0.13);
            \end{tikzpicture}
\caption{\label{fig:cellsofcomplex23}The six $2$-cells of $\CellComplexComplement(2,3)$}
    \end{figure}
        
Our drawing also specifies orientations of the cells that will be discussed in the next section
 (namely, the edges and $2$-cells).
\end{example}

\subsection{A short summary}

We have constructed and described the following objects:
\begin{itemize}[ $\bullet$ ]
\itemsep=0mm 
  \item $\PosetOfStratification(d,n)$: the face poset of the Fox--Neuwirth stratification of $W_n^{\oplus d}$, with minimal element $\hat0$;
  \item $\CellComplexFullSphere(d,n)$: a regular cell complex homeomorphic to $S(W_n^{\oplus d})\cong S^{d(n-1)-1}$, 
  a PL sphere; its face poset is $\PosetOfStratification(d,n){\setminus}\{\hat0\}$;
  \item $\CellComplexFullBall(d,n)$: a regular cell complex homeomorphic to $B(W_n^{\oplus d})$,
   a PL ball, given by $\CellComplexFullSphere(d,n)$ plus one additional $d(n-1)$-cell;
  \item $\SimplicialFullBall(d,n)$: the barycentric subdivision of $\CellComplexFullBall(d,n)$; 
  a simplicial complex, geometrically embedded
           as a star-shaped neighborhood of the origin in $W_n^{\oplus d}$;
          
  \item $\PosetOfComplement(d,n)$: the poset of all strata that lie in the complement of the arrangement (that is, have no $i_j=d+1$), 
  					ordered by reversed inclusion; it is the subposet of $\PosetOfStratification(d,n)^{op}$ given by all $(\sigma,\ii)$ 
					without any index $i_j=d+1$;
  \item $\CellComplexComplement(d,n)$: a regular cell complex of dimension $(d-1)(n-1)$; it is a subcomplex of the dual cell complex to
          $\CellComplexFullSphere(d,n)$; its poset of non-empty faces is $\PosetOfComplement(d,n)$.
  \item $\SimplicialComplement(d,n)$: the barycentric subdivision of $\CellComplexComplement(d,n)$, the order complex
                   $\Delta\PosetOfComplement(d,n)$; a simplicial complex, geometrically embedded
                   into the complement $\overline F(\R^d,n)\subset W_n^{\oplus d}$ 
                   as a subcomplex of the boundary of  the ball $\SimplicialFullBall(d,n)$. 
\end{itemize}

\section{Equivariant Obstruction Theory}

Our task now is to prove that for any $d\ge1$ and $n\ge2$ 
an equivariant map
\[
\CellComplexComplement(d,n) \ \longrightarrow_{\Symm_n}\  S(W_n^{\oplus(d-1)})
\]
exists if and only if $n$ is not a prime power. Here
\begin{compactitem}[ $\bullet$]
    \item   $\CellComplexComplement(d,n)$ is a finite regular CW complex of dimension $M:=(d-1)(n-1)$ on
            which $\Symm_n$ acts freely
            (the action on the $M$-dimensional cells is described explicitly in Theorem~\ref{thm:cell_complex});
    \item   $S(W_n^{\oplus(d-1)})$ is the set of all $(d-1)\times n$ matrices of column sum $0$ and
            sum of the squares of all entries equal to $1$. This is a sphere of dimension $(d-1)(n-1)-1=M-1$,
            on which $\Symm_n$ acts by permutation of columns; this action is not free
            for $n>2$, but it has no stationary points.
\end{compactitem}
We proceed to apply Equivariant Obstruction Theory, according to tom Dieck \cite[Sect.~II.3]{tomDieck:TransformationGroups}:
Since 
\begin{compactitem}[ $\bullet$]
    \item $\CellComplexComplement(d,n)$ is a free $\Symm_n$-cell complex of dimension $M$,
    \item the dimension of the $\Symm_n$-sphere $S(W_n^{\oplus(d-1)})$ is $M-1$, and
    \item $S(W_n^{\oplus(d-1)})$ is $(M-1)$-simple and $(M-2)$-connected,
\end{compactitem}
the existence of an $\Symm_n$-equivariant map is equivalent to the vanishing of the primary
obstruction
\[
\oo = [c_f]\ \in\ H^M_{\Symm_n}\big(\CellComplexComplement(d,n); \pi_{M-1}(S(W_n^{\oplus(d-1)}))\big).
\]
Here $c_f$ denotes the obstruction cocycle associated with a general position equivariant map 
$f:\CellComplexComplement(d,n)\rightarrow W_n^{\oplus(d-1)}$
(cf.~\cite[Def.~1.5, p.~2639]{BlagojevicBlagojevic}).
Its values on the $M$-cells $\check{c}$ are given by the degrees 
\[
c_f(\check{c})\ =\ \deg\big(r\circ f:\partial\check{c}\longrightarrow W_n^{\oplus(d-1)}{\setminus}\{0\}\longrightarrow S(W_n^{\oplus(d-1)})\big),
\]
where $r$ denotes the radial projection. 

The Hurewicz isomorphism gives an isomorphism of the coefficient $\Symm_n$-module with a homology group:
\[
\pi_{M-1}(S(W_n^{\oplus(d-1)})) \ \cong\ H_{M-1}(S(W_n^{\oplus(d-1)});\Z)\ =:\ \ZZ.
\]
As an abelian group this $\Symm_n$-module $\ZZ=\langle\xi\rangle$ is isomorphic to~$\Z$.
The action of the permutations $\tau\in\Symm_n$ on the module $\ZZ$ is given by
\[
  \tau\cdot\xi\ =\ (\sgn \pi)^{d-1}\xi.
\]
Indeed, each transposition $\tau_{ij}$ in  $\Symm_n$
acts on $W_n$ by reflection in the hyperplane $x_i=x_j$. Thus the action of $\Symm_n$ on $W_n$
reverses orientations according to the sign character, and the action on~$W_n^{\oplus(d-1)}$
is given by $\sgn^{d-1}$.

\subsection{Computing the obstruction cocycle}

We will now compute the equivariant obstruction cocycle $c_f$ in the cellular cochain group
\[
C^M_{\Symm_n}(\CellComplexComplement(d,n);\ZZ)
\]
and then show that $c_f$ is a coboundary of an equivariant $(M-1)$-cochain if and only if $n$ is not a power of a prime.
 
To compute the obstruction cocycle, we use the $\Symm_n$-equivariant linear projection map
\[
 f:  W_n^{\oplus d}\ \longrightarrow\ W_n^{\oplus(d-1)}
\]
obtained by deleting the last row from any matrix $(y_1,\dots,y_n)\in W_n^{\oplus d}$ of row sum $0$.
This map clearly commutes with the projection map $\nu:\R^{d\times n}\rightarrow W_n^{\oplus d}$
which subtracts from each column the average of the columns.

\begin{lemma}\label{lemma4.1}
    The linear map $f$ maps all $M$-cells of $\SimplicialComplement(d,n)\subset W_n^{\oplus d}$
    homeomorphically to the same star-shaped neighborhood $\SimplicialFullBall(d-1,n)$ of $0$ in~$W_n^{\oplus(d-1)}$.
    
    The symmetric group $\Symm_n$ acts on this neighborhood by homeomorphisms that
    reverse orientations according to $\sgn^{d-1}$. 
    Thus the $M$-cells and the $(M-1)$-cells of the complex $\CellComplexComplement(d,n)$  can be oriented
    in such a way that the $\Symm_n$-action on the $M$-cells and on the $(M-1)$-cells changes orientations according to $\sgn^{d-1}$, while
    the boundary of any $M$-cell is the sum of all $(M-1)$-cells in its boundary with $+1$ coefficients.
    
    With these orientations, the obstruction cocycle $c_f$ has the value $+1$ on all oriented $M$-cells
    of $\CellComplexComplement(d,n)$.
\end{lemma}

\begin{proof}
    The $M$-cells of $\CellComplexComplement(d,n)$ are given as $\check{c}(\sigma_1\less{d}\sigma_2\less{d}\cdots\less{d}\sigma_n)$
    for some permutation $\sigma\in\Symm_n$. The $(M-1)$-cells are given by $\check{c}(\sigma_1\less{d}\cdots\sigma_i\less{d-1}\sigma_{i+1}\cdots\less{d}\sigma_n)$
    for some $\sigma\in\Symm_n$ and $1\le i<n$.
    
    The vertices of the barycentric subdivision of the cell $\check{c}(\sigma_1\less{d}\sigma_2\less{d}\cdots\less{d}\sigma_n)$ are exactly those
    vertices $v(\sigma_1'\less{{i_1}}\cdots\less{{i_{n-1}}}\sigma_n')$ for which the letters
    $\sigma_j'$ that are separated only by “$\less{d}$”s appear in the same order in $\sigma$.
     
    The projection $f$ deletes the last row, and thus maps the vertex $v(\sigma_1\less{{i_1}}\cdots\less{{i_{n-1}}}\sigma_n)$
    to the vertex with the same symbol $v(\sigma_1\less{{i_1}}\cdots\less{{i_{n-1}}}\sigma_n)$ of $\SimplicialFullBall(d-1,n)$,
    except for reordering of symbols separated only by $\less{d}$.
    Thus exactly the vertices which differ only by reordering letters
    $\sigma_j'$ that are separated only by “$\less{d}$”s are mapped by $f$ to the same vertex of $\SimplicialFullBall(d-1,n)$.
    
    Thus we get that the barycentric subdivision of each maximal cell 
    (which is a simplicial complex, embedded in the higher-dimensional space $W_n^{\oplus d}$)
    is mapped isomorphically onto the star-shaped neighborhood $\SimplicialFullBall(d-1,n)$
    of the origin in~$W_n^{\oplus(d-1)}$ (as depicted for $n=3$ in Figure~\ref{fig:figure3}). In particular, the vertex $v(\sigma_1\less{d}\sigma_2\less{d}\cdots\less{d}\sigma_n)$
    is the only point of the $M$-cell $\check{c}(\sigma_1\less{d}\sigma_2\less{d}\cdots\less{d}\sigma_n)$
    that is mapped to $0\in W_n^{\oplus(d-1)}$.
    
    We interpret $\SimplicialFullBall(d-1,n)$ as the barycentric subdivision of 
    a cellular $M$-ball $\CellComplexFullBall(d-1,n)$ with one $M$-cell. The $(M-1)$-cells in its boundary are
    given by $\check{c}(\sigma_1\less{d}\cdots\sigma_j\less{{d-1}}\sigma_{j+1}\cdots\less{d}\sigma_n)$ with exactly
    one index $i_j=d-1$ and all other $i_k$'s equal to $d$. Clearly there are $2^n-2$ of those, 
    as they are given by the proper nonempty subsets $\{\sigma_1,\dots,\sigma_j\}\subset\{1,\dots,n\}$.
    
    As $f$ induces a surjective cellular map from $\CellComplexComplement(d,n)$ to $\CellComplexFullBall(d-1,n)$ that is a
    homeomorphism restricted to each cell, we can proceed as follows. We fix an orientation of the one
    $M$-cell of $\CellComplexFullBall(d-1,n)$. (This amounts to fixing an orientation of $W_n^{\oplus(d-1)}$.)
    We define orientations of the $(M-1)$-cells in the boundary of the $M$-cell of $\CellComplexFullBall(d-1,n)$ by demanding that
    the cellular boundary of the $M$-cell is given by the formal sum of all the $(M-1)$-cells with $+1$ coefficients.
    We then define the orientation for each of the $(M-1)$-cells of $\CellComplexComplement(d,n)$ by demanding that
    the map $f$, which maps it homeomorphically to an $(M-1)$-cell in the boundary of $\CellComplexFullBall(d-1,n)$,
    preserves the orientation.
\end{proof}

\subsection{When is the obstruction cocycle a coboundary?}

\begin{lemma}\label{lemma4.2}
    Let $b$ be any equivariant $(M-1)$-dimensional integral equivariant cellular cochain on $\CellComplexComplement(d,n)$,
    then the value of its coboundary is 
    \[
    x_1\binom n1 + x_1\binom n2 + \dots +x_{n-1}\binom n{n-1} 
    \]
    on all $M$-cells, for integers $x_1,\dots,x_{n-1}$.
\end{lemma}

\begin{proof}
    As $b$ needs to be equivariant, we get the condition that the values
    are constant on orbits. (Indeed, this is since the boundary
    operator does not introduce any signs, and the symmetric group acts by introducing the
    same signs $\sgn^{d-1}$ both on the $(M-1)$-cells and on the $\Symm_n$-module~$\ZZ$.) 

    Finally, the $\Symm_n$-orbit of the $(M-1)$-cell $\check{c}(\sigma_1\less{d}\cdots\sigma_j\less{{d-1}}\sigma_{j+1}\cdots\less{d}\sigma_n)$
    has size $\binom nj$.
\end{proof}

\begin{proof}[Proof of Theorem~\ref{thm:main}]
    By \cite[Sect.~II.3, pp.~119--120]{tomDieck:TransformationGroups}, the equivariant map exists if and only if
    the cohomology class $\oo =[c_f]$ vanishes, that is, if $c_f$ is the coboundary $c_f=\delta b$
    of some $(M-1)$-dimensional equivariant cochain~$b$. We have seen 
    in Lemmas~\ref{lemma4.1} and~\ref{lemma4.2} that this is the case if and only if 
    \begin{equation}\label{eq:binom}
    x_1\binom n1 + x_1\binom n2 + \dots +x_{n-1}\binom n{n-1} \ =\ 1
    \end{equation}
    has a solution in integers. By the Chinese remainder theorem, this happens if and only if
    the binomial coefficients do not have a non-trivial common factor.
    By Ram's result, as quoted in the introduction and proved below, this holds if and only if $n$ is not a prime power.
\end{proof} 

\begin{proposition}[Ram \cite{Ram1909}]\label{prop:ram}
 For all $n\in \mathbb{N}$ we have
\[
\gcd\Big\{\binom{n}{1},\binom{n}{2},\dots,\binom{n}{n-1}\Big\}\ =\ 
\begin{cases}
    p   &   \text{if }n \text{ is a prime power, }n=p^k,\\
    1   &   \text{otherwise.}
\end{cases}
\]
\end{proposition}

\begin{proof}[Proof (by M. Firsching and P. Landweber)]
For $n=p^k$, $k\geq 1$ and an indeterminate $x$ we have 
$
 (1+x)^{p^k}\equiv 1+x^{p^k} \mod{p},
$
so all the binomial coefficients we consider are divisible by $p$. Taking the equation
\[
 (1+x)^{p^{k-1}}\equiv 1+x^{p^{k-1}} \mod{p}
\]
to the $p$th power, we obtain
\[
 (1+x)^{p^{k}}\equiv (1+x^{p^{k-1}})^p \mod{p^2},
\]
so $\binom{p^k}{p^{k-1}}\equiv\binom{p}{1}\not\equiv0 \mod{p^2}$ and hence 
$\gcd\big\{p^k=\binom{n}{1},\binom{n}{2},\dots,\binom{n}{n-1}\big\}=p$.

If $n$ is not a prime power, let $n = \prod _{i=1}^{m} p_i^{k_i}$ for $m>1$, $k_i>0$ and distinct primes $p_i$.  
Let $p$ be any of the $p_i$ and $k=k_i$, so that $n=p^{k} a$ and $p$ does not divide~$a$.  
Letting $x$ be an indeterminate, we have 
\[
(1+x)^n = (1+x)^{p^{k} a} \equiv (1+x^{p^k})^a \mod{p}
\]
so $\binom{n}{p^k} \equiv \binom{a}{1} \not\equiv 0\mod{p}$.  Therefore  $p_i$ does not divide $\binom{n}{p_i^{k_i}}$ for each $i$. Hence
\[
\gcd\Big\{n=\binom{n}{1},\binom{n}{p_1^{k_1}},\binom{n}{p_2^{k_2}},\hdots,\binom{n}{p_m^{k_m}}\Big\}=1.\]
and thus $\gcd\left\{\binom{n}{1},\binom{n}{2},\hdots,\binom{n}{n-1}\right\}= 1$. 
\end{proof}

\begin{remark}[by P. Landweber]
    Ram's result (Proposition~\ref{prop:ram}) is also useful in the determination of polynomial generators 
    of the complex bordism ring, a graded polynomial ring over $\Z$ with one generator in each positive even dimension.  
    In dimension $2n$ for which $n+1$ is a prime, the bordism class of the complex projective space $\CP^n$ serves 
    as a generator, but in other dimensions one needs to know Ram's result in order to exhibit generators made 
    from complex projective spaces and Milnor hypersurfaces of type $(1,1)$ in products of two complex projective spaces.  
    This result was known and used by Milnor and Novikov in the early 1960's, and is known to other authors working on complex bordism (cf.\ \cite[pp.~249--252]{Milnor:CollectedIII}, \cite[Problem 16-E on p.~196]{Milnor-characteristic}, \cite[Appendix~2]{Novikov1962}), but usually no proof has been quoted or given in this context. 
	One notable source is Lazard \cite[Lemme 3]{Lazard1955}. 
\end{remark}

In the case when $n$ is a power of $2$ the equation (\ref{eq:binom}) has no
solution modulo $2$, which implies that the top Stiefel--Whitney class of
the bundle
\[
W_n \to  F(\R^d,n)\times_{\Symm_n}W_n \to F(\R^d,n)/\Symm_n
\]
is non-trivial. This fact was first proved by Cohen \& Handel \cite[Lemma 3.2, page 203]{CohenHandel1978} in the
case $d=2$,  and by Chisholm \cite[Lemma 3]{Chisholm1979} in the case
when $d$ is a power of $2$. For further
application of this fact in the context of $k$-regular mappings of
Euclidean spaces consult~\cite[Sect.~8.5]{Z129}.

\subsection*{Three corollaries}

Our calculation above implies a complete calculation for the top equivariant
cohomology group of the cell complex $\CellComplexComplement(d,n)$.

\begin{corollary}
    For $d,n\ge2$,
    \[
    H^M_{\Symm_n}\big(\CellComplexComplement(d,n); \pi_{M-1}(S(W_n^{\oplus(d-1)}))\big)\ =\ 
     \langle [c_f]\rangle\ \cong\ 
    \begin{cases}
    \Z/p & \text{if }n=p^k\text{ is a prime power},\\
    0   &  \text{otherwise.}    
    \end{cases}
    \]
\end{corollary}

Indeed, our cell complex model has only one orbit of maximal cells, thus the group of
equivariant $M$-dimensional cochains is isomorphic to $\Z$. Our calculation shows
that the subgroup of coboundaries is generated by the element
$\text{gcd}\{\binom n1,\binom n2,\dots,\binom n{n-1}\}$.

Moreover, we note a stronger version of the non-existence part of Theorem~\ref{thm:main}.

\begin{corollary}
    For $d\ge2$ and prime powers $n=p^k$, there is no $G$-equivariant map 
    \[
    F(\R^d,n)\ \longrightarrow_{G}\  S(W_n^{\oplus(d-1)})
    \]
    for any $p$-Sylow subgroup $G\subset\Symm_n$.
\end{corollary}
 
This follows from our proof for Theorem \ref{thm:main} using transfer; see \cite[Sect.~5.2]{Z129}.
Moreover, in \cite[Thm.~8.3]{Z129} we derive from equivariant cohomology calculations (the Fadell--Husseini index)
that in the case when $n=p$ is a prime, the equivariant map does not exist for the cyclic (Sylow) subgroup $\Z/n\subset\Symm_n$.

Finally, we obtain an extension of a Borsuk--Ulam type theorem by Cohen \& Connett \cite[Thm.~1]{Cohen-Connett},
namely the generalization from a cyclic group of prime order to an arbitrary cyclic group. 
(See \cite[Thm.~8.9]{Z129} for corresponding result for elementary abelian groups, obtained with 
equivariant cohomology methods.)

\begin{corollary}
\label{cor:Coincidence-1}
Let $n\ge2$ and $d\ge2$ be integers, let $X$ be a free Hausdorff $\Z/n$-space, 
and $f\colon X\to\R^d$ be a continuous map.
If $X$ is $(d-1)(n-1)$-connected, then there exist $x\in X$ and $a\in \Z/n$, $a\neq0$, such that
\[
f(x)=f(a\cdot x).
\]
\end{corollary}

\begin{proof}
If $f$ is injective on $\Z/n$-orbits, then we obtain a continuous map  
\[
  \psi: X\longrightarrow F(\R^d,n), \qquad
        x \longmapsto \big(f(0\cdot x),f(1\cdot x),\dots,f((n-1)\cdot x)\big),
\]	
which is $\Z/n$-equivariant if we use the “left shift” cyclic $\Z/n$-action on $F(\R^d,n)$.

Combined with the equivariant retraction  $r: F(\R^d,n)\rightarrow\CellComplexComplement(d,n)$ constructed above, this yields a $\Z/n$-equivariant map
$r\circ\psi:X\rightarrow\CellComplexComplement(d,n)$ from a free $(d-1)(n-1)$-connected $\Z/n$-space to a free $\Z/n$-space of dimension $(d-1)(n-1)$,
which contradicts Dold's theorem \cite{Dold} \cite[Thm.~6.2.6]{Matousek-BU}.
\end{proof}

\subsection*{Acknowledgements}
When the first version of this paper was released on the \texttt{arXiv}, Dev Sinha pointed us to the recent preprints
\cite{GiustiSinha} and~\cite{AyalaHepworth}. In response to this, we decided to change our  
notation to be in line with these papers,
which continue the developments started by Fox \& Neuwirth \cite{FoxNeuwirth62} and Fuks~\cite{Fuks70}.

Thanks to Imre Bárány, Roman Karasev, Wolfgang Lück, and Jim Stasheff for interesting and useful discussions.
In addition, we are grateful to Moritz Firsching and Peter Landweber for many valuable comments and improvements,
including the reference to Ram's result \cite{Ram1909} and the proof we give for it above.
Moreover, we thank the Israel Journal's referee for very insightful and helpful remarks that have
led to various improvements in the exposition.

Thanks to Till Tantau for \texttt{TikZ} and to Miriam Schlöter for the \texttt{TikZ} figures.

We are grateful to Aleksandra and Torsten for their constant support.

\begin{small} 

\providecommand{\noopsort}[1]{}

\end{small}
\end{document}